\documentclass[10pt,twoside,a4paper]{article}
\usepackage[utf8]{inputenc}					
\usepackage[T1]{fontenc}					
\usepackage{float}							
\usepackage[colorlinks=true,linkcolor=blue]{hyperref} 
\usepackage{booktabs}
\usepackage[italian,english]{babel}			

\usepackage{tabularx}
\usepackage[dvipsnames]{xcolor}
\usepackage{fancyvrb}

\usepackage{verbatim}
\usepackage{subfig}
\usepackage{caption}
\usepackage{graphicx}						

\usepackage{color}							

\usepackage{amsmath,amsfonts,amssymb,amscd,lscape,amstext,amsthm}
\usepackage{stix}

\newtheorem{lemma}{\rm \indent LEMMA}[section]
\newtheorem{definition}{\rm \indent D\small EFINITION}[section]
\newtheorem{theorem}[lemma]{\rm \indent T\small HEOREM}
\newtheorem{remark}[lemma]{\rm \indent R\small EMARK}
\newtheorem{proposition}[lemma]{\rm \indent P\small ROPOSITION}
\newtheorem{corollary}[lemma]{\rm \indent C\small OROLLARY}

\newcommand {\R} {\mathbb{R}} 
 \newcommand {\N} {\mathbb{N}}

\newcommand{\la}{\langle}
\newcommand{\ra}{\rangle}

\newcommand{\yy}{\bar{y}}
\newcommand{\hh}{\bar{h}}
\newcommand {\p} {\partial}

\newcommand{\gpiu}{\gamma^+}
\newcommand{\gmeno}{\gamma^-}
\numberwithin{equation}{section}

\begin{document}
\title{\bf{Stability for the Calder\'on's problem for a class of anisotropic conductivities via an ad-hoc misfit functional}}
\author{Sonia Foschiatti \thanks{Dipartimento di Matematica e Geoscienze, Universit\`a degli Studi di Trieste, Italy, sonia.foschiatti@phd.units.it}
	\qquad Romina Gaburro\thanks{Department of Mathematics and Statistics, University of Limerick, Castletroy, Limerick, Ireland, romina.gaburro$@$ul.ie}
\qquad Eva Sincich\thanks{Dipartimento di Matematica e Geoscienze, Universit\`a degli Studi di Trieste, Italy, esincich@units.it}}
\date{}
\maketitle

\begin{abstract}
We address the stability issue in Calder\'on's problem for a special class of anisotropic conductivities of the form $\sigma=\gamma A$ in a Lipschitz domain $\Omega\subset\mathbb{R}^n$, $n\geq 3$, where $A$ is a known Lipschitz continuous matrix-valued function and $\gamma$ is the unknown piecewise affine scalar function on a given partition of $\Omega$. We define an ad-hoc misfit functional encoding our data and establish stability estimates for this class of anisotropic conductivity in terms of both the misfit functional and the more commonly used local Dirichlet-to-Neumann map. 
\end{abstract}

\section{Introduction}

The paper addresses the so-called Calder\'on's inverse conductivity problem of recovering the conductivity $\sigma$ of a body $\Omega\subset\mathbb{R}^n$ by taking measurements of voltage and electric current on its surface $\partial\Omega$. More specifically, the case when the conductivity is
anisotropic and it is \textit{a-priori} known to be of type $\sigma = \gamma A$, where $A$ is a known Lipschitz continuous matrix valued function on $\Omega$ and $\gamma$ is a piecewise-affine unknown function on a given partition of $\Omega$, is considered. It is well known that in absence
of internal sources or sinks, the electrostatic potential $u$ in a conducting body, described by a domain $\Omega\subset{\mathbb R}^n$, is governed
by the elliptic equation

\begin{equation}\label{eq conduttivita'}
   \mbox{div}(\sigma\nabla{u}) = 0 \qquad \mbox{in} \quad \Omega ,
\end{equation}

where the symmetric, positive definite matrix $\sigma(x)=(\sigma_{ij}(x))_{i,j=1}^n$, $x\in\Omega$ represents the (possibly anisotropic) electric conductivity. The inverse conductivity problem consists of finding $\sigma$ when the so called Dirichlet-to-Neumann (D-N) map

$$
\Lambda_{\sigma}:\,{H}^{\frac{1}{2}}(\partial\Omega)\,\ni\,u\vert_{\partial\Omega}\,
\rightarrow\,{\sigma}\nabla{u}\cdot\nu\vert_{\partial\Omega}\in{H}^{-\frac{1}{2}}(\partial\Omega)
$$

is given for any $u \in {H}^{1}(\Omega)$ solution to (\ref{eq conduttivita'}). Here, $\nu$ denotes the unit outer normal to $\partial\Omega$. If measurements can be taken only on one portion $\Sigma$ of $\partial\Omega$, then the relevant map is called the local D-N map ($\Lambda_{\sigma}^{\Sigma}$).


This problem arises in many different fields such as geophysics, known as DC method, medicine, known as Electrical Impedance Tomography (EIT) and non-destructive testing of materials. The first mathematical formulation of the inverse conductivity problem is due to Calder\'{o}n \cite{C}, where he addressed the problem of whether it is possible to determine the (isotropic) conductivity $\sigma = \gamma I$ by the D-N map. This seminal paper opened the way to the solution to the uniqueness issue
where one is asking whether $\sigma$ can be determined by the knowledge of $\Lambda_{\sigma}$ or its local version when measurements are available on a portion of $\partial\Omega$ only.

The case when measurements can be taken over the full boundary has been studied extensively in the past and the fundamental papers \cite{A1}, \cite{Koh-V1}, \cite{Koh-V2}, \cite{N} and \cite{Sy-U} had led the way of solving the problem of uniqueness in the isotropic case. We also recall the uniqueness results of Druskin who, independently from Calder\'{o}n, dealt directly with the geophysical setting of the problem in \cite{D1}-\cite{D3}. His uniqueness result obtained in \cite{D2} was for conductivities described by piecewise constant functions (see also \cite{A-V}). The problem of recovering the conductivity $\sigma$ by local measurements has been treated more recently (see \cite{La-U}, \cite{La-U-T}). In the present paper, we consider the issue of stability in the inverse conductivity problem, therefore we refer to \cite{Bo}, \cite{Che-I-N} and \cite{U} for an overview regarding the issues of uniqueness and reconstruction of the conductivity.

Regarding the stability issue, Alessandrini proved in \cite{A} that, in the isotropic case and dimension $n\geq 3$, assuming \textit{a-priori} bounds on $\sigma$ of the form $\|\sigma\|_{H^s(\Omega)}\leq E$ , $s>\frac{n}{2}+2$, leads to a continuous dependance of $\sigma$ in $\Omega$ upon $\Lambda_{\sigma}$ of logarithmic type. We also refer to \cite{B-B-R}, \cite{B-F-R} and \cite{Liu} for subsequent results in this direction. Even though stability at the boundary $\partial\Omega$ is of Lipschitz type (see \cite{A-G}, \cite{A-G1}), Mandache \cite{Ma} showed that in the interior of $\Omega$, the inconvenient logarithmic type of stability is the best possible, in any dimension $n\geq 2$, under \textit{a-priori} smoothness assumptions on $\sigma$. It seems therefore reasonable to think that, in order to restore stability in a really (Lipschitz) stable fashion, one needs to replace in some way the \textit{a-priori} assumptions expressed in terms of regularity bounds with \textit{a-priori} pieces of information of a different type that suit the underlying physical problem. Alessandrini and Vessella showed in \cite{A-V} that when $\sigma$ is isotropic and piecewise constant on a given partition of $\Omega$, then Lipschitz stability can be restored in terms of the local D-N map (conditional stability). Rondi \cite{R} proved that the Lipschitz constant has an exponential behaviour with rispect to the number of subdomain of the partition. From a medical imaging point of view, the partition of $\Omega$ may represent different volumes occupied by different tissues or organs and one can think that their geometrical configuration is given by means of other imaging modalities such as MRI. We also recall \cite{A-dH-G-S}, \cite{Be-Fr}, \cite{Be-Fr-V}, \cite{Be-Fr-Mo-Ro-Ve}, \cite{RS1} and \cite{A-dH-G-S1}, \cite{Be-dH-Q}, \cite{Be-dH-F-S}, \cite{RS2} where similar Lipschitz stability results have been obtained for the classical and fractional Calder\`on's problem, the Lam\'{e} parameters and for a Schr\"{o}dinger type of equation.

In this paper we address the issue of stability in Calder\`on's problem in presence of anisotropy. This choice is motivated by the fact that anisotropy appears quite often in nature. Most tissues in the human body are anisotropic. In the theory of homogenization, anisotropy results as a limit in layered or fibrous structures such as rock stratum or muscle, as a result of crystalline structure or of deformation of an isotropic material. In the geophysical context, in 1920, Conrad Schlumberger \cite{S} recognized that anisotropy may affect geological formations' electrical properties and anisotropic effects when measuring electromagnetic fields in geophysical applications have been studied ever since. Individual minerals are typically anisotropic but rocks composed of them can appear to be isotropic. 

From a mathematical point of view, the inverse problem with anisotropic conductivities is an open problem. Since Tartar's observation \cite{Koh-V1b} that any diffeomorphism of $\Omega$ which keeps the boundary points fixed has the property of leaving the D-N map unchanged, whereas $\sigma$ is modified, different lines of research have been pursued. One direction has been to find the conductivity up to a diffeomorphism which keeps the boundary fixed (see \cite{As-La-P}, \cite{B}, \cite{La-U}, \cite{La-U-T}, \cite{Le-U}, \cite{N} and \cite{Sy}). Another direction has been the one to formulate suitable \textit{a-priori} assumptions (possibly fitting some physical context) which constrain the structure of the unknown anisotropic conductivity. For instance, one can formulate the hypothesis that the directions of anisotropy are known while some scalar space dependent parameter is not. Along this line of reasoning, we mention the results in \cite{A}, \cite{A-G}, \cite {A-G1}, \cite{G-L}, \cite{G-S}, \cite{Koh-V1} and \cite{L}. We also refer to \cite {A-dH-G}, \cite{As-La-P}, \cite{B}, \cite{Ds-K-S-U}, \cite{Ds-Ku-L-S}, \cite{F-K-R}, \cite{La-U} and for related results in the anisotropic case and to \cite{A-dH-G}, \cite{Gr-La-U1} and \cite{Gr-La-U2} for examples of non-uniqueness. 

Here, we follow this second direction by \textit{a-priori} assuming that the conductivity is of type

	\begin{equation}\label{a priori info su sigmaj}
	\sigma(x)=\sum_{m=1}^{N}\gamma_{m}(x)\,\chi_{D_m}(x)\,A(x),\qquad\mbox{for any}\:x\in\Omega,
	\end{equation}

where $\gamma_{m}(x)$ is an unknown affine  scalar function on $D_m$, $A$ is a known Lipschitz continuous matrix-valued function on $\Omega$ and $\{D_m\}_{m=1}^N$ is a given partition of $\Omega$ (the precise assumptions on $\sigma$, $A$ and $\{D_m\}_{m=1}^N$ are given in Subsections \ref{subsec assumption domain} and \ref {apriori sigma}). Allowable partitions for our machinery to work include, in the geophysical setting, models of layered media and bodies with multiple inclusions. The ill-posed nature of the EIT inversion is aggravated the deeper one tries to image inside a body $\Omega$ \cite{Nag-U-W}, where EIT image resolution becomes quite poor (see \cite{Ga-H}), leading to blurry images. Thus, in a geophysical context for example, it becomes difficult to recognise individual thin sediments and rock layers or fractures in the deep subsurface, but the `average' effect at large scale of fine layering and fracturing are still shown as equivalent anisotropic media. It seems therefore reasonable to model the conductivity $\sigma$ within each layer $D_m$ by an anisotropic conductivity $\sigma_m$ to make it up for the finer layering structure within $D_m$ that otherwise might have been neglected the deeper one goes inside $\Omega$ due to poor resolution. 

In order to introduce the misfit functional, consider two anisotropic conductivities $\sigma^{(1)}$ and $\sigma^{(2)}$ of type \eqref{a priori info su sigmaj}. If measurements are locally taken on an open portion $\Sigma\subset\partial\Omega$, we conveniently enlarge the physical domain $\Omega$ to an augmented domain $\widetilde\Omega$ and consider Green's functions $G_i$ for $\mbox{div}(\sigma^{(i)}\nabla\cdot)$ in $\widetilde\Omega$, for $i=1,2$, with poles $y,z\in\widetilde\Omega\setminus\bar\Omega$ respectively. Hence we express the error in the measurements corresponding to $\sigma^{(1)}$ and $\sigma^{(2)}$ by means of the misfit functional
\begin{equation}\label{misfit2}
\mathcal{J}(\sigma^{(1)},\sigma^{(2)})=\int_{D_y\times D_z} \left|S_{\mathcal{U}_0}(y,z)\right|^2dy dz,
\end{equation}
where $D_y$, $D_z$ are suitably chosen sets compactly contained in $\widetilde\Omega\setminus\bar\Omega$ and $S_{\mathcal{U}_0}(y,z)$ is defined by the surface integral
\begin{equation}\label{misfit1}
S_{\mathcal{U}_0}(y,z) = \int_{\Sigma} \left[ G_2(\cdot,z)\sigma^{(1)}(\cdot)\nabla G_1(\cdot,y)\cdot\nu - G_1(\cdot,y)\sigma^{(2)}(\cdot)\nabla G_2(\cdot,z)\cdot\nu\right]\:dS.
\end{equation}
We have obtained the following stability estimate of H\"older type:
\begin{equation}\label{stabilita' globale intro}
	\|\sigma^{(1)}-\sigma^{(2)}\|_{L^{\infty}(\Omega)}\leq C
	\left(\mathcal{J}(\sigma^{(1)},\sigma^{(2)})\right)^{1\slash 2},
	\end{equation}
where $C>0$ is a constant that depends on the \textit{a-priori} information only. The augmented domain $\widetilde\Omega$ is chosen in such a way that $G_1(\cdot, y)\big|_{\partial\Omega}$, $G_2(\cdot, z)\big|_{\partial\Omega}$ are supported in $\Sigma$ in the trace sense, hence belonging to the domain of the local D-N maps  $\Lambda_{\sigma_i}^{\Sigma}$, $i=1,2$ (see Section \ref{localDN} for the formal definitions of the local D-N map and the appropriate spaces). Therefore, not only \eqref{stabilita' globale intro}, together with the well-known Alessandrini's identity \cite{A1}, implies a Lipschitz stability estimate of $\sigma$ in terms of the more commonly used local D-N map in the mathematical literature, but it also indicates that the set of measurements $\left\{G(\cdot , y)\big|_{\partial\Omega}\right\}$, with $y,\in\widetilde\Omega\setminus\bar\Omega$ is enough to stably determine $\sigma$. A Lipschitz stability estimate in terms of $\Lambda^{\Sigma}_{\sigma}$ was obtained in \cite{G-S} for the case $\sigma = \gamma A$, with $\gamma$ piecewise constant instead. The piecewise affine parametrizations considered in the present work tie in well with the finite elements method for
computations. With the stability estimate \eqref{misfit2} at hand, one can apply certain iterative methods for reconstruction within a subspace of piecewise affine functions with a starting model at a distance less than the radius of convergence to the unique solution \cite{A-dH-F-G-S}, \cite{F-A-Ba-dH-G-S}, \cite{dH-Q-S} and \cite{dH-Q-S 1}. This radius is known to be roughly inversely proportional to the stability constant appearing in the estimate. More importantly, we can iteratively construct the best piecewise affine approximation for a given domain partition. Since the stability constant will grow at least exponentially with the number of subdomains in the partition \cite{R}, the radius of convergence shrinks accordingly. One can expect accurate piecewise affine approximations with relatively less subdomains (compared to the piecewise constant case of \cite{G-S}) to describe the subsurface, noting that the domain partition need not be uniform and may show a local refinement, and hence our result provides the necessary insight for developing a practical approach with relatively minor prior information.

To the best of our knowledge a first stability estimate in terms of an ad-hoc misfit functional was achieved in the mathematical literature in \cite{A-dH-F-G-S} in the context of the Full Waveform Inversion. Such an estimate proved to be key for the implementation and reliability of a reconstruction procedure (see  \cite{A-dH-F-G-S, F-A-Ba-dH-G-S}) based on the use of Cauchy data only,  being the latter independent on the availability of the Dirichlet to Neumann map. In the more recent result in \cite{Fa-dH-S} an ad-hoc misfit functional has been introduced in the context of imaging elastic media.

We also observe that another advantage of choosing the misfit functional over the local D-N map (even if available) to model the measurements error in EIT is motivated by its potentially simpler numerical implementation, compared to the computation of the norm of bounded linear operators between $H^{\frac{1}{2}}$ spaces and their duals. Moreover, the misfit functional could also provide, again in the context of a possible numerical reconstruction of $\sigma$, additional features compared to the more traditional least-squares approach, allowing, in particular, for a distinction between the computational and the observational measurements. This is due to the introduction of the possibly distinct sets $D_y$ and $D_z$ that can almost be arbitrarily chosen outside the physical domain $\Omega$. For example, $D_y$ could be an arbitrarily chosen set for the numerical data acquisition for the sake of the simulations, where $D_z$ could model a more realistic set that fits the geometric disposition of the electrodes in the actual measurements acquisition. Hence, in the discrete setting, such distinction can potentially require minimal information about the observational acquisition geometry of the electrodes employed for the observational measurements. This is due to the definition of the misfit functional that does not compare simulations and observations directly, but it rather compares products of observed and simulated measurements. Note also that with a slight modification, our arguments can apply when the local Neumann-to-Dirichlet (N-D) map is available instead, see for instance the discussion in \cite{A-G1}.

The paper is organized as follows. In Section \ref{sec2} we introduce the main assumptions on the domain $\Omega$ and the anisotropic conductivity $\sigma$. Section \ref{sec2} contains the formal definitions of the local D-N map (subsection \ref{localDN}), the misfit functional (sunsection \ref{misfitsec}) and the statement of our main result (Theorem \ref{teorema principale}). A Lipschitz stability estimate in terms of the local D-N map follows as a straightforward consequence (Corollary \ref{corollary}). Section \ref{sec3} is devoted to the introduction of some technical tools of asymptotic estimates for the Green function (Proposition \ref{teorema stime asintotiche}) and propagation of smallness (Proposition \ref{proposizione unique continuation finale}) needed for the machinery of the proof of Theorem \ref{teorema principale}. The proof of Theorem \ref{teorema principale} and Corollary \ref{corollary} are also contained in this section. Section \ref{technical} contains the proofs of Proposition \ref{teorema stime asintotiche} and Proposition \ref{proposizione unique continuation finale}.

\section{Misfit functional and the main result}\label{sec2}

\setcounter{equation}{0}

\subsection{Assumptions about the domain $\Omega$}\label{subsec assumption domain}

For $n\geq 3$,
a point $x\in \mathbb{R}^n$ will be denoted by
$x=(x',x_n)$, where $x'\in\mathbb{R}^{n-1}$ and $x_n\in\mathbb{R}$.
Moreover, given a point $x\in \mathbb{R}^n$,
we will denote with $B_r(x), B_r'(x')$ the open balls in
$\mathbb{R}^{n},\mathbb{R}^{n-1}$ respectively centred at $x$ and $x'$ with radius $r$
and by $Q_r(x)$ the cylinder
\[Q_r(x)=B_r'(x')\times(x_n-r,x_n+r).\]

Set $B_r=B_r(0)$, $Q_r=Q_r(0)$, the positive real half space $\R^n_+ = \{ (x',x_n)\in \R^n\,:\, x_n>0 \}$, the positive semisphere centred at the origin $B^+_r = B_r\cap \R^n_+$, the positive semicylinder $Q^+_r = Q_r\cap\R^n_+$. Similar definitions for $\R^n_-$, $B^-_r$ and $Q^-_r$.

Let us recall a couple of definitions concerning the regularity of the boundary of the domain.
\begin{definition}\label{def Lipschitz boundary}
	Let $\Omega$ be a bounded domain in $\mathbb R^n$. A portion
	$\Sigma$ of $\partial\Omega$ is of Lipschitz class with constants
	$r_0,L>0$ if for each point $P\in\Sigma$ there exists a rigid
	transformation of coordinates under which $P$ coincides with the origin and
	$$\Omega\cap Q_{r_0}=\left\{x\in Q_{r_0}\,:\,x_n>\varphi(x')\right\},$$
	where $\varphi$ is a Lipschitz function on $B'_{r_0}$ such that $\varphi(0)=0$ and $	\|\varphi\|_{C^{0,1}(B'_{r_0})}\leq Lr_0.$
\end{definition}

\begin{definition}\label{flat portion}
	Let $\Omega$ be a domain in $\mathbb R^n$. A subset $\Sigma$ of
	$\partial\Omega$ is a flat portion of size $r_0$
	if for each point $P\in\Sigma$ there exists a rigid transformation of coordinates under which $P$ coincides with the origin and
	\begin{eqnarray}\label{local flat cond}
	\Sigma\cap{Q}_{r_{0}} =\left\{x\in
	Q_{r_0}\,:\,x_n=0\right\},\quad &&
	\Omega\cap {Q}_{r_{0}}=\left\{x\in
	Q_{r_0}\,:\,x_n>0\right\}.\nonumber
	\end{eqnarray}
\end{definition}


From now on, we will consider $\Omega\subset \R^n$, $n\geq 3$ as a bounded, measurable domain with boundary $\p\Omega$ of Lipschitz class with positive constants $r_0, L$ as in Definition \ref{def Lipschitz boundary} and satisfying

\begin{equation}\label{misura Omega}
|\Omega|\leq N r_0^n,
\end{equation}

where $|\Omega|$ denotes the Lebesgue measure of $\Omega$. Moreover, we assume that there exists a partition of bounded subdomains $D=\{D_m\}_{m=1}^N$ contained in $\Omega$ such that the following conditions hold:
\begin{enumerate}
	\item $D_m$ for $m=1,\dots,N$ are connected, pairwise non-overlapping subdomains with boundaries $\p D_m$ which are of Lipschitz class with constants $r_0$, $L$
	\item $\overline{\Omega} = \bigcup_{m=1}^{N}\overline{D}_m$;
	\item (\textit{Chain of subdomains.}) First, we assume that there exists one region, let us call it $D_1$, such that the intersection
		$\partial{D}_1\cap\Sigma$ contains a \emph{flat} portion
		$\Sigma_1$ of size $r_0\slash 3$ (see Definition \ref{flat portion}) and that for every $i\in\{2,\dots , N\}$ there exists a collection of indices $m_1,\dots ,
		m_K\in\{1,\dots , N\}$ such that $D_{m_1}=D_1$ and $D_{m_K}=D_i$ and the subdomains are pairwise disjoint. Secondly, we assume that, for every fixed sub-index $k=1,\dots , K$ of the chain, the intersection $\partial{D}_{m_{k}}\cap \partial{D}_{m_{k+1}}$ contains a \emph{flat} portion $\Sigma_{m_{k+1}}$ of size $r_0\slash 3$ such that
	$\Sigma_{m_{k+1}}\subset\Omega$ for $k=1,\dots,K-1$. 
		Finally, for each of these flat sub-portions $\Sigma_{m_{k+1}}$, $k=1,\dots , K-1$, there exist a point $P_{k+1}\in\Sigma_{m_{k+1}}$ and a rigid transformation of coordinates under which $P_{k+1}$ coincides with the origin	and
	\begin{eqnarray}
		\Sigma_{m_{k+1}}\cap{Q}_{r_{0}/3} &=&\left\{x\in
		Q_{r_0/3}\,:\,x_n=0\right\},\nonumber\\
		D_{m_k}\cap {Q}_{r_{0}/3} &=&\left\{x\in
		Q_{r_0/3}\,:\,x_n<0\right\},\nonumber\\
		D_{m_{k+1}}\cap {Q}_{r_{0}/3} &=&\left\{x\in
		Q_{r_0/3}\,:\,x_n>0\right\}.\nonumber
	\end{eqnarray}
	Later, we will add a domain $D_0\subset \mathbb{R}^n\setminus\overline{\Omega}$ so that, when indexing the chain of subdomains, we agree that
	$D_{m_0}=D_0$.
\end{enumerate}

\subsection{A-priori information on the anisotropic conductivity $\sigma$}\label{apriori sigma}
Our stability result for the Calder\'on inverse problem concerns a special family of anisotropic conductivities $\sigma$. Let us describe in details their form. The conductivities $\sigma(x)=\{\sigma_{ij}(x)\}$ are real-valued, symmetric $n\times n$ matrices such that $\sigma\in L^{\infty}(\Omega,Sym_n)$ and have the form 
\begin{subequations}
	\begin{eqnarray}\label{apriorisigma}
	&&\sigma(x)=\gamma(x) A(x)\quad\label{conductivity 1}\\
	&&\gamma(x)=\sum_{m=1}^{N}\gamma_{m}(x)\chi_{D_m}(x),\quad
	\gamma_{m}(x)=s_m+S_m\cdot x,\quad
	\mbox{for any}\:x\in\Omega,\label{conductivity 2},
	\end{eqnarray}
\end{subequations}
where the scalars $s_m\in\mathbb{R}$ and the vectors $S_m\in\mathbb{R}^n$, $m=1,\dots,N$ are the unknowns, $A(x)$ is a known fixed matrix and $D=\{D_m\}_{m=1}^N$ is the known partition of $\Omega$ introduced in Section \ref{subsec assumption
	domain}. Furthermore,
\begin{itemize}
	\item[a)] the scalar functions $\gamma_{m}$ are bounded, piecewise linear and there is a positive constant $\bar{\gamma}>1$ such that
	\begin{eqnarray}\label{apriorigamma}
	\bar{\gamma}^{-1}\le \gamma_m(x)\le  \bar{\gamma} , \qquad &&\mbox{for
		any}\:m=1,\dots N,\,\,\,\mbox{for
		any}\:x\in \Omega;
	\end{eqnarray}
	\item[b)] the matrix $A(x)$ satisfies the following Lipschitz continuity condition: there exists a constant $\bar{A}>0$ such that $\|A\|_{\mathcal{C}^{0,1}(\Omega)}\le \bar{A}$;
	
	\item[c)] The matrix $\sigma$ is positive definite and there exists a constant $\lambda>1$ such that
	\begin{eqnarray}\label{ellitticita'sigma}
	\lambda^{-1}\vert\xi\vert^{2}\leq A(x)\,\xi\cdot\xi\leq\lambda\vert\xi\vert^{2},\qquad 
	&&\mbox{for a.e.}\:x\in\Omega,\,\,\,\mbox{for every}\:\xi\in\mathbb{R}^{n}.
	\end{eqnarray}
\end{itemize}

\begin{definition}
	The set of positive constants $\{N, r_0, L, \lambda, \bar{\gamma}, \bar{A}, n\}$ with $N\in\mathbb{N}$ and the space dimension $n\geq 3$, is called the \textit{a-priori data}.  
\end{definition}
In the paper several constants depending on the \textit{a-priori data} will appear. In order to simplify our notation, we will denote them by $C, C_1,C_2\dots$, avoiding in most cases to point out their specific dependence on the a priori data which may vary from case to case.

%


\subsection{The local Dirichlet-to-Neumann map}\label{localDN}
By now, assume simply that $\Omega$ is a bounded domain with $\p\Omega$ of Lipschitz class. Since Dirichlet data are different from zero on a small portion $\Sigma\subset \p\Omega$, we introduce a suitable trace space for the formulation of the local Dirichlet-to-Neumann map.
\begin{definition}
	Let $\Sigma$ be a non-empty (flat) open portion of $\p\Omega$. The subspace of $H^{1\slash 2}(\partial\Omega)$ of trace functions which are compactly supported in $\Sigma$ is defined as
	\begin{equation}\label{Hco}
	H^{1\slash 2}_{co}(\Sigma)=\left\{f\in
	H^{1\slash 2}(\p\Omega)\,:\,\textnormal{supp}
	\:f\subset\Sigma\right\}.
	\end{equation}
	The trace space $H^{1\slash 2}_{00}(\Sigma)$ is the closure of $H^{1\slash 2}_{co}(\Sigma)$ with respect to the $H^{1\slash2}(\p\Omega)$-norm. We denote by $H^{-1\slash 2}_{00}(\Sigma)$ the dual of the trace space $H^{1\slash 2}_{00}(\p\Omega)$. 
\end{definition}

\begin{definition}\label{DN}
	The local Dirichlet-to-Neumann (DN) map associated with $\sigma$ and $\Sigma$ is the operator
	\begin{align}\label{mappaDN}
	\Lambda^{\Sigma}_{\sigma}\,:\,H^{1\slash 2}_{00}(&\Sigma)\,\rightarrow \,H^{-1\slash 2}_{00}(\Sigma)\\
	&g\,\,\,\, \mapsto\,\, \sigma\,\nabla u\cdot \nu\big|_{\Sigma},\nonumber
	\end{align}
	where $\nu$ is the unit outward normal of $\p \Omega$ and $u\in{H}^{1}(\Omega)$ is the weak solution to the boundary value problem
	\begin{displaymath}
	\left\{ \begin{array}{ll} \mbox{div}\left(\sigma(\cdot)\nabla u\right)=0, &
	\textrm{$\textnormal{in}\quad\Omega$},\\
	u=g, & \textrm{$\textnormal{on}\quad{\partial\Omega}.$}
	\end{array} \right.
	\end{displaymath}
	
	The map \eqref{mappaDN} can be identified with the bilinear form $H^{1\slash 2}_{00}(\Sigma)\times H^{1\slash 2}_{00}(\Sigma)\rightarrow \R$ defined by
	\begin{equation}\label{def DN locale}
	\la\Lambda_{\sigma}^{\Sigma}\:g,\:\eta\ra\:=\:\int_{\:\Omega}
	\sigma(x)\: \nabla{u}(x)\cdot\nabla\varphi(x)\:dx,
	\end{equation}
	where $\eta\in H^{1\slash 2}_{00}(\Sigma)$ and $\varphi\in H^{1}(\Omega)$ is any function such that $\varphi\vert_{\Sigma}=\eta$. In \eqref{def DN locale} the bracket $\la\cdot,\:\cdot\ra$ denotes the $L^{2}(\partial\Omega)$-pairing
	between $H^{1\slash 2}_{00}(\Sigma)$ and its dual
	$H^{-1\slash 2}_{00}(\Sigma)$.
\end{definition}
For simplicity, we will denote by
$\parallel\cdot\parallel_{*}$ the $\mathcal{L}(H^{1\slash 2}_{00}(\Sigma),H^{-1\slash 2}_{00}(\Sigma))$-norm of the Banach space of
bounded linear operators from $H^{1\slash 2}_{00}(\Sigma)$ to $H^{-1\slash 2}_{00}(\Sigma)$.


\subsection{Misfit functional}\label{misfitsec}
To begin with, we introduce the Green function $G$ in an augmented domain $\widetilde\Omega$ as follows. From the assumptions on the domain $\Omega$ (Section \ref{subsec assumption domain}) there is a point $P_1\in \Sigma$ that coincides with the origin, up to a rigid transformation of coordinates. For simplicity, let us assume that the locally flat portion $\Sigma_1$ coincides with the entire portion $\Sigma$. Let us define the domain $D_0\subset \R^n\setminus \overline{\Omega}$ as
\begin{equation}\label{dtilde}
D_0=\left\{x\in(\mathbb{R}^n\setminus\overline{\Omega})\cap B_{r_0}\:\bigg|\:|x_i|<\frac{r_0}{3},\:i=1,\dots , n-1,\:\:-\frac{r_0}{3}<x_n<0\right\},
\end{equation}
and such that
\[
\p D_0\cap\p\Omega\,\,\subset\subset\, \Sigma.
\]
We define the augmented domain $\widetilde{\Omega}$ as the set 
\begin{equation}\label{aumenteddomain}
\widetilde\Omega\,\,=\,\,\overset{\circ}{\overline{\Omega\cup D_0}}.
\end{equation}
It turns out that $\widetilde\Omega$ is of Lipschitz class with constants $\frac{r_0}{3}$ and
$\tilde{L}$, where $\tilde{L}$ depends on $L$ only.

Denote 
$$(D_0)_r = \left\{\,x\in D_0\::\:\mbox{dist}(x,\p D_0)>r\,\,\right\}, \qquad r\in \left(0,\frac{r_0}{6}\right).$$
Finally, we introduce two sets contained in $D_0$: the sets $D_y$ and $D_z$ which are compactely supported in $D_0$, i.e. $D_y,D_z\subset\subset D_0$. In the following sections, we might identify these sets with the set $(D_0)_r$, but in general, thay can be freely chosen in $D_0$.

Consider two anisotropic conductivities $\sigma^{(i)}$, $i=1,2$ as in Section \ref{apriori sigma}. Without loss of generality, we can extend them to the augmented domain $\widetilde{\Omega}$ by setting their value equal to the identity matrix on $D_0$, so that they are of the form
\begin{displaymath}
\begin{array}{ll}
\sigma^{(i)}(x)=\gamma^{(i)}(x)A(x), &\textrm{$\textnormal{for any }x\in\Omega$},\\
\sigma^{(i)}|_{D_0}=I, &\gamma^{(i)}|_{D_0}=1.
\end{array}
\end{displaymath}
We denote with the same symbol $\sigma$ the extended conductivity.

For every $y\in D_0$, the Green's function $G_i(\cdot,y)$ associated to $L_i = \mbox{div}(\sigma^{(i)}(\cdot)\nabla\cdot)$ and $\widetilde\Omega$ with pole $y$, is the weak solution to the Dirichlet problem

\begin{equation}\label{greensystem}
\left\{ \begin{array}{ll}
\mbox{div}(\sigma^{(i)}(\cdot)\nabla G_i(\cdot,y))=-\delta(\cdot-y) &\textrm{$\textnormal{in}\quad\widetilde\Omega$},\\
G_i(\cdot,y)=0 &\textrm{$\textnormal{on}\quad\p\widetilde\Omega$},
\end{array}
\right.
\end{equation}
where $\delta(\cdot - y)$ is the Dirac distribution centred at $y$.

We recall the following properties for the Green's functions (see \cite{Lit-St-W}):
\begin{equation*}
G(x,y)=G(y,x),\qquad \forall x\neq y,
\end{equation*}
and
\begin{equation}\label{boundgreen}
0<G(x,y)<C |x-y|^{2-n},\qquad \forall x\neq y.
\end{equation}
For $(y,z)\in D_y\times D_z$, define the following surface integral
\begin{equation}\label{misfit1}
S_{\mathcal{U}_0}(y,z) = \int_{\Sigma} \left[ G_2(x,z)\,\sigma^{(1)}(x)\nabla G_1(x,y)\cdot\nu - G_1(x,y)\,\sigma^{(2)}(x)\nabla G_2(x,z)\cdot\nu\right]\:dS(x).
\end{equation}
We define the \textit{misfit functional} as the quantity
\begin{equation}\label{misfit2}
\mathcal{J}(\sigma^{(1)},\sigma^{(2)})=\int_{D_y\times D_z} \left|S_{\mathcal{U}_0}(y,z)\right|^2dy\, dz.
\end{equation}


\subsection{Stability estimate}\label{mainresultsec}
In previous works (see \cite{A-dH-G-S}, \cite{A-V},\cite{G-S}), Lipschitz stability estimates have been established for piecewise constant and piecewise linear isotropic conductivities and a certain class of anisotropic conductivities respectively, in terms of the local Dirichlet-to-Neumann map. Here, we extend these results to the class of anisotropic conductivities defined in Section \ref{apriori sigma}. First, we determine a bound to the $L^{\infty}$-norm of the difference between two anisotropic conductivities in terms of the square root of the misfit functional introduced above. Then, we derive a Lipschitz stability result in terms of the local D-N map.
\begin{theorem}\label{teorema principale}
	Let $\Omega$ be a bounded domain as in assumptions \ref{subsec assumption domain}.
	Let $\sigma^{(1)}$ and $\sigma^{(2)}$ be two anisotropic conductivities as in assumptions \ref{apriori sigma}, i.e. of the form
	\begin{equation}\label{a priori info su sigma}
		\sigma^{(i)}(x)=\sum_{m=1}^{N}\gamma_{m}^{(i)}(x)\chi_{D_m}(x)A(x),\quad\quad\mbox{for any}\:x\in\Omega,\,\,\:i=1,2,	
	\end{equation}
	where $D=\{D_m\}_{m=1}^N$ is the chain of subdomains as in assumptions \ref{subsec assumption domain}, $A(x)$ is the known Lipschitz matrix and $\gamma_{m}^{(i)}(x)$ are the piecewise-affine functions given by the formula
	\begin{equation*}
	\gamma_{m}^{(i)}(x)=s^{(i)}_m+S^{(i)}_m\cdot x,\qquad x\in D_m,
	\end{equation*}
	for $s^{(i)}_m\in\mathbb{R}$ and $S^{(i)}_m\in\mathbb{R}^n$. 
	Then there exists a positive constant $C$ such that 
	
	\begin{equation}\label{stabilita' globale}
	\|\sigma^{(1)}-\sigma^{(2)}\|_{L^{\infty}(\Omega)}\leq C
	\left(\mathcal{J}(\sigma^{(1)},\sigma^{(2)})\right)^{1\slash 2},
	\end{equation}
	where $C$ depends on the a priori data only.
\end{theorem}
From this result, it follows a Lipschitz stability estimate in terms of the local D-N maps.
\begin{corollary}\label{corollary}
	Assume that the hypothesis of Theorem \ref{teorema principale} hold, then
	\begin{equation}\label{corollaryequation}
	\|\sigma^{(1)}-\sigma^{(2)}\|_{L^{\infty}(\Omega)}\leq C \|\Lambda_{\sigma^{(1)}}^{\Sigma}-\Lambda_{\sigma^{(2)}}^{\Sigma}\|_{*},
	\end{equation}
	where $C>0$ is a constant depending on the a-priori data only.
\end{corollary}
%

\begin{remark}
	From now on, as we deal with two different anisotropic conductivities $\sigma^{(i)}$, $i=1,2$, we will simply denote with the symbol $\Lambda_i$ the local DN map $\Lambda^{\Sigma}_{\sigma^{(i)}}$.
\end{remark}

\section{Proof of the main result}\label{sec3}
The proof of Theorem \ref{teorema principale} is based on an argument that combines asymptotic estimates for the Green's function of the elliptic operator $\mbox{div}(\sigma(\cdot)\nabla\cdot)$ (Proposition \ref{teorema stime asintotiche}), together with a result of unique continuation (Proposition \ref{proposizione unique continuation finale}). In this section we introduce these technical results (proved in Section \ref{technical}), then we prove Theorem \ref{teorema principale} and Corollary \ref{corollary}.
 
\subsection{Technical tools}

\subsubsection{\normalsize{Behaviour of Green's function near interfaces}}\label{subsection Green function}



We shall denote with
\begin{eqnarray}\label{fondamentalsol}
\Gamma(x,y)=\frac{1}{n(2-n)\omega_n}|x-y|^{2-n},\quad &&\omega_n= \frac{2\:\pi^{n\slash 2}}{n\:\Gamma(n\slash 2)},
\end{eqnarray}
the fundamental solution for the Laplace operator (here $\omega _n$ denotes the volume of the unit ball in $\mathbb{R}^n$).

Let $\{D_m\}_{m=0}^K$, $K\in\{1,\dots,N\}$ be the chain of subdomains as in assumptions \ref{subsec assumption domain}, $\{\Sigma_m\}_{m=1}^K$ be the corresponding sequence of flat portions with special points $P_1,\dots,P_K$. Moreover, let $\nu(P_{m+1})$ denotes the unit normal to $\p D_{m}$ at the point $P_{m+1}$ pointing outside $D_m$.
\begin{proposition}\label{teorema stime asintotiche}({Asymptotic estimates})
	Fix an index $m\in\{0, \dots, K-1\}$, then there exist constants $\alpha, \theta_1, \theta_2, 0<\alpha,\theta_1,\theta_2<1$ and $C_1,C_2,C_3>0$ depending on the a priori data only and a suitable constant $C_4>1$ such that the following inequalities hold true for every $x\in B_{\frac{r_0}{C_4}}(P_{m+1})\cap
	D_{m+1}$ and every $y=P_{m+1}-r \nu(P_{m+1})$, where $r\in
	(0,\frac{r_0}{C_4})$ 
	
	\begin{eqnarray}
	&&\left|G(x,y) -
	\frac{2}{\gamma_{m}(P_{m+1})+\gamma_{m+1}(P_{m+1})}\Gamma(Jx,Jy)\right| \le  {C_1}|x-y|^{3-n-\alpha} \label{asyfun},\\
	&&\left|\nabla_x G(x,y) -
	\frac{2}{\gamma_{m}(P_{m+1})+\gamma_{m+1}(P_{m+1})}\nabla_x\Gamma(Jx,Jy)\right| \le  {C_2}|x-y|^{1-n +\theta_1}\label{asydernor},\\
	&&\left|\nabla_y\nabla_xG(x,y) -
	\frac{2}{\gamma_{m}(P_{m+1})+\gamma_{m+1}(P_{m+1})}\nabla_y\nabla_x\Gamma(Jx,Jy)\right| \le  {C_3}|x-y|^{-n+\theta_2} \label{asydernorgrad}\ .
	\end{eqnarray}
	where $J$ is the positive definite matrix
	$J=\sqrt{A(P_{m+1})^{-1}}$.
\end{proposition}

\subsubsection{{\normalsize Quantitative unique continuation}}\label{uniquecontinuation}

For any number $b>0$, define the concave,
non decreasing function $\omega_{b}(t)$ on $(0,+\infty)$ as
\begin{displaymath}
\omega_{b}(t)=\left\{ \begin{array}{ll} 2^{b}e^{-2}|\log t|^{-b},
&\quad
t\in (0,e^{-2}),\\
e^{-2}, &\quad t\in[e^{-2},+\infty)
\end{array} \right.
\end{displaymath}
We recall (see (4.34) and (4.35) in \cite{A-V}) that
\begin{eqnarray}\label{property1}
(0,+\infty)\ni t \rightarrow \ t\omega_{b}\left(\frac{1}{t}\right)\quad &&\mbox{is a non-decreasing function}
\end{eqnarray}
and for any $\beta\in(0,1)$ we have that
\begin{eqnarray}\label{property2}
\omega_{b}\left(\frac{t}{\beta}\right)\le |\log e\beta^{-1/2}|^b\omega_{b}(t)\ ,\quad &&\omega_{b}(t^{\beta})\le \left(\frac{1}{\beta} \right)^b\omega_{b}(t)\ .
\end{eqnarray}

Furthermore, we shall denote the iterative compositions of $\omega$ as
\begin{equation*}
\omega_{b}^{(1)}=\omega_b\ ,\qquad \omega_{b}^{(j)}=\omega_{b}\circ \omega_{b}^{(j-1)}\ \ j=2,3, \dots
\end{equation*}
and we set $\omega_{b}^{(0)}(t)=t^{b}$ for $0<b<1$. 

Fix a chain of subdomains $\{D_{m}\}_{m=0}^K$ as in assumptions \ref{subsec assumption domain} for the domain $\widetilde{\Omega}$. Set
\begin{equation}\label{Wk}
\mathcal{W}_k = \bigcup_{m=0}^{k}D_{m},\qquad \mathcal{U}_k = \widetilde\Omega\setminus\overline{\mathcal{W}_k},\quad\textnormal{for}\:k=0,\dots,K.
\end{equation}
\begin{definition}
	For any $y,z\in \mathcal{W}_k$, define the singular solution
	\begin{equation*}
	S_{\mathcal{U}_k}(y,z)=\int_{\mathcal{U}_k}\big(\sigma^{(1)}(\cdot)-\sigma^{(2)}(\cdot)\big)\nabla G_1(\cdot,y)\cdot\nabla G_2(\cdot,z),\quad \textnormal{for}\:k=0,\dots,K.
	\end{equation*}
\end{definition}
The set $\{S_{\mathcal{U}_k}(y,z)\}_{k=0}^K$ is a family of real-valued functions which satisfies the following inequality:
\begin{equation}\label{remark on SU}
|S_{\mathcal{U}_k}({y},z)|\leq C
\|\sigma^{(1)}-\sigma^{(2)}\|_{L^{\infty}(\Omega)}\left(d(y)d(z)\right)^{1-\frac{n}{2}},
\quad\mbox{for\:every}\:y,z\in\mathcal{W}_k,
\end{equation}
where $d(y)=\mbox{dist}(y,\mathcal{U}_k)$ and $C$ is a positive constant depending on $\lambda$ and $n$ only.

One can prove (see \cite{A-V}) that for every $y,z\in\mathcal{W}_k$ with $k=0,\dots,K$, the functions $S_{\mathcal{U}_k}(\cdot,z),
S_{\mathcal{{U}}_k}(y,\cdot)$ belongs to $H^1_{loc}(\mathcal{W}_k)$ and are weak solutions, respectively, to
\begin{equation*}
\textnormal{div} \left(\sigma^{(1)}(\cdot)\nabla
S_{\mathcal{U}_k}(\cdot,z)\right)=0,\qquad \textnormal{div}
\left(\sigma^{(2)}(\cdot) \nabla S_{\mathcal{U}_k}(y,\cdot)\right)=0 \,\quad\mbox{in}\:\mathcal{W}_k.
\end{equation*}

We introduce the following parameters:
\begin{eqnarray}\label{parameters}
	& &\beta=\arctan{\frac{1}{L}},\quad\beta_1 = \arctan{\left(\frac{\sin\beta}{4}\right)},\quad\lambda_1=\frac{r_0}{1+\sin\beta_1},\nonumber\\
	& & \rho_1=\lambda_1\sin\beta_1,\quad a=\frac{1-\sin\beta_1}{1+\sin\beta_1},\nonumber\\
	& & \lambda_m=a\lambda_{m-1},\quad \rho_m = a\rho_{m-1},\quad\textnormal{for\:every}\:m\geq 2,\nonumber\\
	& & d_m=\lambda_m-\rho_m,\quad m\geq 1.
\end{eqnarray}

Notice that $d_m=r_0 a^m, \ 0<a<1$.

Choose $l\in \N$, fix a point $\bar{y}\in \Sigma_{m+1}$, then define 
\begin{eqnarray}
w=w_{l}(\bar{y})=\bar{y}-\lambda_{l}\nu(\bar{y}), \qquad \mbox{for every} \ l\ge 1 \ ,
\end{eqnarray}
where $w$ is a point into the domain $D_{m}$ near the interface $\Sigma_{m+1}$. 
For a given $r\in (0,d_1]$ define the function
\begin{eqnarray}\label{hbar}
\bar{h}(r)=\min\{l\in\mathbb{N}\::\:d_{l}\leq r\}.
\end{eqnarray}
For successive estimates, it is important to point out the following inequality:
\begin{eqnarray}\label{Sh2}
\log\left(\frac{r}{d_1}\right)^C\le \bar{h}(r)-1\le \log\left(\frac{r}{d_1}\right)^C +1, \quad &&C=\frac{1}{|\log a|}.
\end{eqnarray}

The following estimate for $S_{\mathcal{U}_k}(y,z)$ holds true, for any $k=1,\dots , K$.

\begin{proposition}\label{proposizione unique continuation finale}({Estimates of unique continuation})
	Suppose that for a positive number $\varepsilon_0$ and $r>0$ we have
	\begin{equation}\label{estim0}
	\left|S_{\mathcal{U}_k}(y,z)\right|\leq
	r_0^{2-n}\varepsilon_0,\quad \mbox{for every}\: (y,z)\in
	(D_0)_{r}\times(D_0)_{r},
	\end{equation}
%
%
	then the following inequalities hold true for every $r\in (0,d_1]$
	\begin{equation}\label{estim1}
	\left|S_{\mathcal{U}_k}\left(w_{\bar{h}}(Q_{k+1}),w_{\bar{h}}(Q_{k+1})\right)\right|
	\leq
	C_1^{\bar{h}}(E+\varepsilon_0)\left(\omega_{1/C}^{(2k)}\left(\frac{\varepsilon_0}{E+\varepsilon_0}\right)\right)
	^{\left(1/C\right)^{\bar{h}}},
	\end{equation}
	\begin{equation}\label{estim2}
	\left|\partial_{y_j}\partial_{z_i}S_{\mathcal{U}_k}\left(w_{\bar{h}}(Q_{k+1}),w_{\bar{h}}(Q_{k+1})\right)\right|
	\leq
	C_2^{\bar{h}}(E+\varepsilon_0)\left(\omega_{1/C}^{(2k)}\left(\frac{\varepsilon_0}{E+\varepsilon_0}\right)\right)
	^{\left(1/C\right)^{\bar{h}}},
	\end{equation}
	for any $i,j=1,\dots , n$, where $Q_{k+1}\in\Sigma_{k+1}\cap B_{\frac{r_0}{8}}(P_{k+1})$,
	$w_{\bar{h}(r)}(Q_{k+1})=Q_{k+1}-\lambda_{\bar{h}(r)}\nu(Q_{k+1})$, with $\lambda_{\bar{h}(r)}$ as above,
	$\nu(Q_{k+1})$ is the exterior unit normal to $\partial{D}_{k}$  at the point $Q_{k+1}$ pointing outside $D_k$ and $C_1,C_2>0$
	depend on the a-priori data only.
	
\end{proposition}

\subsection{Proof of Theorem \ref{teorema principale} and the Corollary \ref{corollary}}

\begin{proof}[Proof of \textit{Theorem \ref{teorema principale}.}] 
	First, notice that
	\[
	\|\sigma^{(1)}-\sigma^{(2)}\|_{L^{\infty}(\Omega)}\le \|\gamma^{(1)}-\gamma^{(2)}\|_{L^{\infty}(\Omega)}\,\,\bar{A},
	\]
	where $\bar{A}$ is the Lipschitz constant from assumptions \ref{apriori sigma}. Let $D_K$ be the subdomain of $\Omega$ such that
		\begin{align*}
	\|\gamma^{(1)}-\gamma^{(2)}\|_{L^{\infty}(\Omega)} 
	=\|\gamma_K^{(1)}-\gamma_K^{(2)}\|_{L^{\infty}(D_K)}.
	\end{align*}
	
	Then, inequality \eqref{stabilita' globale} will follow from 
	\begin{equation}\label{sigmaKestimate}
	\|\gamma_K^{(1)}-\gamma_K^{(2)}\|_{L^{\infty}(D_K)}\le C \Big(\mathcal{J}(\sigma^{(1)},\sigma^{(2)})\Big)^{1\slash 2},
	\end{equation}
	for $C>1$ a positive constant depending on a priori estimates.
	
	To prove \eqref{sigmaKestimate}, we find convenient , as previously stated, to work in the augmented domain $\widetilde{\Omega}$ as in \eqref{aumenteddomain}, where $D_0$ is the domain defined in \eqref{dtilde}, on which we have defined the extended conductivity $\sigma^{(i)}$ for $i=1,2$ by setting $\sigma^{(i)}|_{D_0}=I$. Recalling that $D_K$ is the subdomain of $\Omega$ where the maximum of $|\gamma^{(1)} - \gamma^{(2)}|$ is reached, let $D_0,D_1,\dots ,D_K$ be the chain of subdomains as in Section \ref{subsec assumption domain} and let $\Sigma_1,\dots,\Sigma_K$ be the corresponding flat portions. Set
	\begin{align}\label{deltak}
	\varepsilon_0=\Big(\mathcal{J}(\sigma^{(1)},\sigma^{(2)})\Big)^{1\slash 2},\qquad &E=\|\gamma_K^{(1)}-\gamma_K^{(2)}\|_{L^{\infty}(D_K)},\\
	\delta_k=\|\gamma^{(1)}-\gamma^{(2)}\|_{L^{\infty}(\mathcal{W}_k)}, \quad &\mbox{for }k=1,\dots,K.
	\end{align}
	Given a differentiable function $f$ on a domain $\Omega$, we can split its differential as
	\[
	Df(x)=D_T f(x) + \p_{\nu}f(x), \qquad \mbox{for }x\in \Sigma_k, \, k=1,\dots,K,
	\]
	where $D_{T}f$ is the $n-1$ dimensional vector of the tangential partial derivatives of $f$ on $\Sigma_k$ and $\partial_\nu f$ denotes the normal partial derivative of $f$ on $\Sigma_k$, respectively for $k=1,2,\dots,K$. 
	
	Fix $0<r_2<r_1$ such that $\Sigma_k\cap B_{r_1}(P_k)\neq \emptyset$ for $k=1,2,\dots,K$. We observe that the norm $\|\gamma^{(1)}_k - \gamma^{(2)}_k \|_{L^{\infty}(D_k)}$ can be estimated in terms of the quantities
	\begin{eqnarray}\label{coefficientbounds}
	\|\gamma^{(1)}_k - \gamma^{(2)}_k \|_{L^{\infty}(\Sigma_k\cap B_{r_1}(P_k))}\quad\mbox{and}\quad &&\left| \partial_{\nu}(\gamma^{(1)}_k - \gamma^{(2)}_k)(P_k)\right|.
	\end{eqnarray}
	In fact, fix an orthonormal basis $\{e^k_j\}_{j=1,\dots, n-1}$ which generates the hyperplane containing the flat part $\Sigma_k$. Set
	\begin{equation*}
	\alpha_k+\beta_k\cdot x = \big(\gamma^{(1)}_k -\gamma^{(2)}_k\big)(x),\qquad x\in D_k.
	\end{equation*}
	 If we evaluate $\big(\gamma^{(1)}_k -\gamma^{(2)}_k\big)$ at the points $P_k+r_2e^k_j$, $j=1,\dots, n-1$, it follows that
	\begin{equation*}
	\left|\alpha_k + \beta_k\cdot \big(P_k+r_2 e^k_j\big)\right| \leq
	|\alpha_k+\beta_k\cdot P_k|+r_1 \sum_{j=1}^{n-1}|\beta_k\cdot e^k_j|
	\leq  C\|\gamma^{(1)}_k -\gamma^{(2)}_k\|_{L^{\infty}(\Sigma_k\cap B_{r_1}(P_k))}.
	\end{equation*}
	Next, notice that
	\begin{equation*}
	|\beta_k\cdot\nu|=\left|\p_{\nu}(\gamma^{(1)}_k-\gamma^{(2)}_k)(P_k)\right|.
	\end{equation*}
	In conclusion, for $k=1,\dots,K$,
	\begin{equation*}
	|\alpha_k|+|\beta_k|\leq C\left(\|\gamma^{(1)}_k -\gamma^{(2)}_k\|_{L^{\infty}(\Sigma_k\cap B_{r_1}(P_k))} + \left|\partial_{\nu}(\gamma^{(1)}_k-\gamma^{(2)}_k)(P_k)\right|\right).
	\end{equation*}
	Hence, our task will be to estimate the quantities introduced in \eqref{coefficientbounds} for $k=1,\dots,K$ in terms of the function $\omega_{1\slash C}$ introduced in Section \ref{uniquecontinuation}, $\varepsilon_0$ and $E$.
	
	\subsubsection{Boundary estimates}
	Let us start from the case $k=1$. We will prove the following estimate:
	
	
	\begin{equation}\label{casek1}
	\|\gamma^{(1)}_1 -\gamma^{(2)}_1\|_{L^{\infty}(\Sigma_1\cap B_{r_1}(P_1))}+\left|\partial_{\nu}(\gamma^{(1)}_1-\gamma^{(2)}_1)(P_1)\right|\leq  C(\varepsilon_0+E)\omega_{1/C}^{(0)}\left(\frac{\varepsilon_0}{\varepsilon_0+E}\right).
	\end{equation}
	For every $y,z\in (D_0)_{r}$, by Green formula the following equalities hold:
	
	\begin{align}\label{Alessandrini 1}
	&\int_{\Sigma} \left[ G_2(\cdot,z)\,\,\sigma^{(1)}(\cdot)\nabla G_1(\cdot,y)\cdot\nu - G_1(\cdot,y)\sigma^{(2)}(\cdot)\nabla G_2(\cdot,z)\cdot\nu\right]\:dS
	=\\
	&=\int_{\Omega}(\sigma^{(1)}-\sigma^{(2)})(\cdot)\nabla G_1(\cdot,y)\cdot\nabla G_2(\cdot,z),\nonumber
	\end{align}
	and
	\begin{align}\label{Alessandrini 2}
	&\int_{\Sigma} \left[ \p_{z_n} G_2(\cdot,z)\,\,\sigma^{(1)}(\cdot)\nabla \p_{y_n} G_1(\cdot,y)\cdot\nu - \p_{y_n} G_1(\cdot,y)\,\,\sigma^{(2)}(\cdot)\nabla \p_{z_n} G_2(\cdot,z)\cdot\nu\right]\:dS
	=\\
	&=\int_{\Omega}(\sigma^{(1)}-\sigma^{(2)})(\cdot)\,\nabla\partial_{y_n}G_1(\cdot,y)\cdot\nabla\partial_{z_n} G_2(\cdot,z),\nonumber
	\end{align}
	where $G_1(\cdot,y)$ and $G_2(\cdot,z)$ are weak solutions to the problem \eqref{greensystem}.
Since $S_{\mathcal{U}_0}(y,z)$ and $\p_{y_n}\p_{z_n}S_{\mathcal{U}_0}(y,z)$ are weak solutions to the following equation
\[
\mbox{div}\Big(\sigma^{(1)}(\cdot)\nabla S_{\mathcal{U}_0}(\cdot,z)\Big) + \mbox{div}\Big(\sigma^{(2)}(\cdot)\nabla S_{\mathcal{U}_0}(y,\cdot)\Big)=0, \quad \mbox{in }D_y\times D_z,
\]
we can apply a result of local boundedness for weak solutions of a uniformly elliptic operator (see \cite[Chapter 8]{G-T}) that allows us to bound the supremum of $S_{\mathcal{U}_0}(y,z)$ by its $L^2$-norm as follows:
\begin{align}\label{su0bound}
\sup\limits_{(y,z)\in (D_y)_r\times (D_z)_r} |S_{\mathcal{U}_0}(y,z)|\leq C\Big(\int_{D_y\times D_z} |S_{\mathcal{U}_0}(y,z)|^2 dy\,dz\Big)^{1\slash 2} = C\Big(\mathcal{J}(\sigma^{(1)},\sigma^{(2)})\Big)^{1\slash 2},
\end{align}
where $C$ depends on $n$, $\lambda$, $|\Omega|$ and $r\in (0,r_0/6)$.

Let $\rho_0=r_0\slash C_4$, where $C_4$ is the constant introduced in Theorem \ref{teorema stime asintotiche}. Let $r\in(0,d_2]$ and define the point $w=w(P_1)=P_1 - \tau\nu(P_1)$ where $\nu(P_1)$ is the unit outward normal of $\p D_1$ at the point $P_1$ and $\tau=\lambda_{\bar{h}(r)}=a^{\bar{h}-1}\lambda_1, \,\hh=\hh(r)$ is defined in \eqref{hbar}.

Set $y=z=w$, split the right hand side of \eqref{Alessandrini 1} into the sum of two integrals $I_1(w)$ and $I_2(w)$:
\[
S_{\mathcal{U}_{0}}(w,w)=I_1(w)+I_2(w),
\]
where
\begin{eqnarray}
I_1(w)&=&\int_{B_{\rho_0}(P_1)\cap D_1}(\gamma_1^{(1)}-\gamma_1^{(2)})(\cdot)A(\cdot)\nabla G_1(\cdot,w)\cdot
\nabla G_2(\cdot,w),\nonumber\\
I_2(w)&=&\int_{\Omega\setminus (B_{\rho_0}(P_1)\cap
	D_1)}(\sigma^{(1)}-\sigma^{(2)})(\cdot)
\nabla G_1(\cdot,w)\cdot\nabla G_2(\cdot,w).\nonumber
\end{eqnarray}
The integral $I_2(w)$ can be easily estimate using \cite[Proposition 3.1]{A-V} as
\begin{equation}\label{caccio}
|I_2(w)|\leq CE\rho_0^{2-n},
\end{equation}
Let us estimate $I_1(w)$ from below in terms of $\|\gamma^{(1)}_1-\gamma^{(2)}_1\|_{L^{\infty}(\Sigma_1\cap B_{r_1}(P_1))}$. Let $\overline{x}\in \overline{\Sigma_1\cap B_{r_1}(P_1)}$ be such that
\[
(\gamma^{(1)}_1-\gamma^{(2)}_1)(\overline{x})=\|\gamma^{(1)}_1-\gamma^{(2)}_1\|_{L^{\infty}(\Sigma_1\cap B_{r_1}(P_1))}.
\]
Since $(\gamma^{(1)}_1 - \gamma^{(2)}_1)(x)=\alpha_1+\beta_1\cdot x$, 
\begin{eqnarray}\label{I1first}
I_1(w)  & = & \int_{B_{\rho_0}(P_1)\cap D_1}(\gamma^{(1)}_1-\gamma^{(2)}_1)(\overline{x})A(x)\nabla G_1(x,w)\cdot
\nabla G_2(x,w)\:dx +\nonumber\\
&+& \int_{B_{\rho_0}(P_1)\cap D_1} \beta_1\cdot(x-\overline{x})A(x)
\nabla G_1(x,w)\cdot\nabla G_2(x,w)\:dx,
\end{eqnarray}
which leads to
\begin{eqnarray}\label{I1second}
|I_1(w)|  & \geq &  \left|\int_{B_{\rho_0}(P_1)\cap D_1}(\gamma^{(1)}_1-\gamma^{(2)}_1)(\bar{x})\:A(x)\:\nabla G_1(x,w)\cdot
\nabla G_2(x,w)\right|-\nonumber\\
&-& \bar{A}\int_{B_{\rho_0}(P_1)\cap D_1}|\beta_1\cdot (x-\bar{x})|\:|\nabla G_1(x,w)|\:
|\nabla G_2(x,w)|\:dx.
\end{eqnarray}
If we set $\tilde{c}^{(1)}=\frac{2}{1+\gamma_1^{(1)}(P_1)}$ and $\tilde{c}^{(2)}=\frac{2}{1+\gamma_1^{(2)}(P_1)}$, by adding and subtracting the fundamental solution $\tilde{c}^{(i)}\Gamma$ we have
\begin{eqnarray}\label{I1third}
|I_1(w)| &\geq& \left|\int_{B_{\rho_0}(P_1)\cap D_1}(\gamma^{(1)}_1-\gamma^{(2)}_1)(\bar{x})\:A(x)\:\tilde{c}^{(1)}\:\tilde{c}^{(2)}\:|\nabla\Gamma(Jx,Jw)|^2\right|-\nonumber\\
&-& \int_{B_{\rho_0}(P_1)\cap D_1}|(\gamma^{(1)}_1-\gamma^{(2)}_1)(\bar{x})|\:|A(x)\nabla(G_1(x,w)-\tilde{c}^{(1)}\Gamma(Jx,Jw))\cdot\nabla(G_2(x,w)-\tilde{c}^{(2)}\Gamma(Jx,Jw))|dx\nonumber\\
&-& \bar{A}\int_{B_{\rho_0}(P_1)\cap D_1}|(\gamma^{(1)}_1-\gamma^{(2)}_1)(\bar{x})|\:|\nabla(G_1(x,w)-\tilde{c}^{(1)}\Gamma(Jx,Jw))|\tilde{c}^{(2)}|\nabla\Gamma(Jx,Jw)|dx\nonumber\\
&-& \bar{A}\int_{B_{\rho_0}(P_1)\cap D_1}|(\gamma^{(1)}_1-\gamma^{(2)}_1)(\bar{x})|\:\tilde{c}^{(1)}\:|\nabla\Gamma(Jx,Jw)||\nabla(G_2(x,w)-\tilde{c}^{(2)}\Gamma(Jx,Jw))|\:|dx\nonumber\\
&-& \int_{B_{\rho_0}(P_1)\cap D_1}|\beta_1\cdot (x-\bar{x})|\,|A(x)\nabla \Gamma(Jx,Jw)|\cdot\nabla \Gamma(Jx,Jw)|\:dx.
\end{eqnarray}
Now, up to a change of coordinate we can suppose that $P_1$ is the origin $O$. Let us apply the asymptotic estimate \eqref{asydernorgrad} to \eqref{I1third} for $J=\sqrt{A^{-1}(0)}$, it follows that

\begin{eqnarray}
|I_1(w)|&\geq & \|\gamma^{(1)}_1-\gamma^{(2)}_1\|_{L^{\infty}(\Sigma_1\cap B_{r_1})}C\lambda^{-1}\int_{B_{\rho_0}\cap D_1}|\nabla_x\Gamma(Jx,Jw)|^2\:dx-\nonumber\\
& &-C\:E\int_{B_{\rho_0}\cap D_1} |\nabla_x\Gamma(Jx,Jw)|\:|x-w|^{\theta_1+1-n}\:dx -\nonumber\\
& &-C\:E\int_{B_{\rho_0}\cap D_1} |x-w|^{2\theta_1+2-2n}\:dx -\nonumber\\
& &- C\:E\int_{B_{\rho_0}\cap D_1} |x-\overline{x}|\:|x-w|^{2-2n}\:dx,\nonumber
\end{eqnarray}
where the $C>0$ depends on the \textit{a-priori} data only. By definition \eqref{fondamentalsol}, we can express explicitly the fundamental solution $\Gamma$ inside the integrals and obtain:
\begin{eqnarray}\label{I1estimates}
|I_1(w)| &\geq& \|\gamma^{(1)}_1-\gamma^{(2)}_1\|_{L^{\infty}(\Sigma_1\cap B_{r_1})}C\lambda^{-1}\int_{B_{\rho_0}\cap D_1}\frac{|J^2(x-w)|}{|J(x-w)|^{n}}^2\:dx-\nonumber\\
& &-C\:E\int_{B_{\rho_0}\cap D_1} \frac{|J^2(x-w)|}{|J(x-w)|^n}\:|x-w|^{\theta_1+1-n}\:dx -\nonumber\\
& &-C\:E\int_{B_{\rho_0}\cap D_1} |x-w|^{2\theta_1+2-n}\:dx -\nonumber\\
& &- C\:\int_{B_{\rho_0}\cap D_1}|\beta_1|\:|x-\overline{x}|\:|x-w|^{2-2n}\:dx.
\end{eqnarray}
By estimating the integrals in \eqref{I1estimates} with respect to the parameter $\tau$, we can bound $|I_1(w)|$ from below as follows:
\begin{eqnarray}\label{I1estimatestau}
|I_1(w)| &\geq & \|\gamma^{(1)}_1-\gamma^{(2)}_1\|_{L^{\infty}(\Sigma_1\cap B_{r_1})}C\tau^{2-n} -C\:E\:\tau^{2-n+\theta_1}-C\tau^{2-n+2\theta_1}-C\:E\:\tau^{3-n}.
\end{eqnarray}
By \eqref{su0bound} and \eqref{caccio}, it follows that
\[
|I_1(w)|\leq |S_{\mathcal{{U}}_0}(w,w)|+|I_2(w)|\leq C\:\varepsilon_0\tau^{2-n} + C\:E\:\rho_0^{2-n},
\]
which leads to the following estimate for the conductivity:
\begin{equation*}
\|\gamma^{(1)}_1-\gamma^{(2)}_1\|_{L^{\infty}(\Sigma_1\cap B_{r_1}(P_1))}\tau^{(2-n)}\leq C\:\varepsilon_0\tau^{2-n} + C\:E\:\rho_0^{2-n} + C\:E\:\tau^{2-n+\theta_1}+C\tau^{2-n+2\theta_1}+C\:E\:\tau^{3-n}.\nonumber
\end{equation*}
Dividing by $\tau^{2-n}$ both sides and for $\tau\rightarrow 0^+$, we obtain

\begin{equation}\label{first estimate}
\|\gamma^{(1)}_1-\gamma^{(2)}_1\|_{L^{\infty}(\Sigma_1\cap B_{r_1}(P_1))}\leq C\varepsilon_0.
\end{equation}

Let us estimate $|\p_{\nu}(\gamma^{(1)}_1-\gamma^{(2)}_1)(P_1)|$. From \eqref{Alessandrini 2}, for $y=z=w$ as above, we split again the $nth$ partial derivative of the singular solution as follows:

\begin{equation}
\p_{y_n}\p_{z_n}S_{\mathcal{U}_0}(w,w)=\bar{I}_1(w)+\bar{I}_2(w),
\end{equation}
where
\[
\bar{I}_1(w)=\int_{B_{\rho_0}(P_1)\cap D_1}(\gamma^{(1)}_1-\gamma^{(2)}_1)(\cdot)A(\cdot)\nabla\partial_{y_n}G_1(\cdot,w)\cdot
\nabla\partial_{z_n}G_2(\cdot,w),
\]

\[
\bar{I}_2(w)=\int_{\Omega\setminus (B_{\rho_0}(P_1)\cap D_1)}(\sigma^{(1)}-\sigma^{(2)})(\cdot)\nabla\partial_{y_n}G_1(\cdot,w)\cdot
\nabla\partial_{z_n}G_2(\cdot,w).\]

With a similar argument as in \eqref{caccio} one can determine an upper bound for $\bar{I}_2$ of the form
\begin{equation}
|\bar{I}_2(w)|\leq CE\rho_0^{-n},
\end{equation}	
where $C$ depends on the \textit{a-priori} data.
Notice that for any point $x\in B_{\rho_0}(P_1)\cap D_1$, the following equality holds
\[
(\gamma^{(1)}_1-\gamma^{(2)}_1)(x)=(\gamma^{(1)}_1-\gamma^{(2)}_1)(P_1) + (D_T(\gamma^{(1)}_1-\gamma^{(2)}_1)(P_1))\cdot (x-P_1)' + (\partial_{\nu}(\gamma^{(1)}_1-\gamma^{(2)}_1)(P_1))(x-P_1)_n,
\]
Proceeding as in \eqref{I1first} and \eqref{I1second}, 

\begin{eqnarray*}
	|\bar{I}_1(w)| &\geq& \left|\int_{B_{\rho_0}(P_1)\cap D_1}(\partial_{\nu}(\gamma^{(1)}_1-\gamma^{(2)}_1)(P_1))(x-P_1)_n A(\cdot)\nabla\partial_{y_n}G_1(\cdot,w)\cdot
	\nabla\partial_{z_n}G_2(\cdot,w)\right|\nonumber\\
	& &-\int_{B_{\rho_0}(P_1)\cap D_1}|(D_T(\gamma^{(1)}_1-\gamma^{(2)}_1)(P_1))\cdot (x-P_1)'|\:|A(\cdot)\nabla\partial_{y_n}G_1(\cdot,w)\cdot
	\nabla\partial_{z_n}G_2(\cdot,w)|\nonumber\\
	& &-\int_{B_{\rho_0}(P_1)\cap D_1}|(\gamma^{(1)}_1-\gamma^{(2)}_1)(P_1)|\:|A(\cdot)\nabla\partial_{y_n}G_1(\cdot,w)\cdot
	\nabla\partial_{z_n}G_2(\cdot,w)|.
\end{eqnarray*}
Up to a rigid transformation, we can assume that $P_1$ coincides with the origin $O$ of the coordinate system. Using a similar technique as in \eqref{I1third} and by Theorem \ref{teorema stime asintotiche}, this leads to

\begin{eqnarray}
|\bar{I}_1(w)|
& &\geq
|\partial_{\nu}(\gamma^{(1)}_1-\gamma^{(2)}_1)(O)|C\int_{B_{\rho_0}\cap
	D_1}|\nabla_x\partial_{y_n}\Gamma(Jx,Jw)|^2 |x_n|-\nonumber\\
& &- C\bigg\{E\int_{B_{\rho_0}\cap
	D_1}|\partial_{y_n}\nabla_x\Gamma(Jx,Jw)|\:|x-w|^{\theta_2-n}|x_n|+\nonumber\\
& &+E\int_{B_{\rho_0}\cap
	D_1}|x-w|^{\theta_2-2n}|x_n|\bigg\}-\nonumber\\
& &-\int_{B_{\rho_0}\cap D_1}|D_T(\gamma^{(1)}_1-\gamma^{(2)}_1)|\:|x'|\:|\nabla\partial_{y_n}G_1(\cdot,w)|\:|
\nabla\partial_{z_n}G_2(\cdot,w)|-\nonumber\\
& & -\int_{B_{\rho_0}\cap D_1}\!\!|(\gamma^{(1)}_1-\gamma^{(2)}_1)(O)|\:|\nabla\partial_{y_n}G_1(\cdot,w)|\:|
\nabla\partial_{z_n}G_2(\cdot,w)|.
\end{eqnarray}

By \eqref{first estimate}, we derive the following lower bound:
\begin{eqnarray*}
	|\bar{I}_1(w)| &\geq &
	|\partial_{\nu}(\gamma^{(1)}_1-\gamma^{(2)}_1)(O)|C\int_{B_{\rho_0}(P_1)\cap D_1}|x-w|^{1-2n}-\noindent\\
	&-&C\:\bigg\{E\int_{B_{\rho_0}\cap	D_1}|x-w|^{1-2n+\theta_2}
	-\int_{B_{\rho_0}\cap D_1}|x-w|^{2-2n+\theta_2}-\nonumber\\
	&-&\varepsilon_0\int_{B_{\rho_0}\cap D_1}\:|x-w|^{1-2n} -\varepsilon_0\int_{B_{\rho_0}\cap D_1}\:|x-w|^{-2n}\bigg\},
\end{eqnarray*}
which leads to
\begin{equation}\label{I11}
|\partial_{\nu}(\gamma^{(1)}_1-\gamma^{(2)}_1)(O)|\tau^{1-n}\le |I_1(w)| + C\:\Big(\varepsilon_0\tau^{-n} + E\tau^{1-n+\theta_2}\Big).
\end{equation}

By unique continuation \eqref{estim2},
\begin{eqnarray}\label{I1}
|\bar{I}_1(w)|&\leq& |\p_{y_n}\p_{z_n}S_{\mathcal{{U}}_0}(w,w)|+|I_2(w)|\\
&\leq& C\:\varepsilon_0\tau^{-n} + C\:E\:\rho_0^{-n},\nonumber
\end{eqnarray}


Thus, by combining together \eqref{I11} and \eqref{I1}, it follows that
\begin{equation*}
|\partial_{\nu}(\gamma^{(1)}_1-\gamma^{(2)}_1)(O)|\tau^{1-n} \le  C \Big(\varepsilon_0\tau^{-n}+ E \rho_0^{-n}+\varepsilon_0\tau^{-n}+E\tau^{1-n+\theta_2}\Big),
\end{equation*}
which leads to
\begin{equation*}
|\partial_{\nu}(\gamma^{(1)}_1-\gamma^{(2)}_1)(O)| \le   C \left( \varepsilon_0 \tau^{-1}+E\tau^{\theta_2}\right).
\end{equation*}
Finally, optimizing the right hand side with respect to $\tau$, the estimate is given by the following inequality
\begin{equation*}
|\partial_{\nu}(\gamma^{(1)}_1-\gamma^{(2)}_1) (O)| \le   C \varepsilon_0^{\frac{\theta_2}{\theta_2+1}}(E+\varepsilon_0)^{\frac{1}{1+\theta_2}},
\end{equation*}
so that \eqref{casek1} is proved.

\subsubsection{Interior estimates}

We show that from the case $k=1$ we obtain the following estimate for the case $k=2$:
\begin{eqnarray}
& &\|\sigma^{(1)}_2-\sigma^{(2)}_2\|_{L^{\infty}(\Sigma_2\cap B_{r_1}(P_2))}\leq  \! C(\varepsilon_0+E)\left(\omega_{1/C}^{(3)}\left(\frac{\varepsilon_0}{\varepsilon_0+E}\right)\right)^
{\frac{1}{C}},\label{lipschitz stability gamma Sigma 2}\\
& &\left|\partial_{\nu}(\sigma^{(1)}_2-\sigma^{(2)}_2)(P_2)\right| \leq  C(\varepsilon_0+E)\left(\omega_{1/C}^{(4)}\left(\frac{\varepsilon_0}{\varepsilon_0+E}\right)\right)^{\frac{1}{C}}.\label{pointwise gamma normal Sigma 2}
\end{eqnarray}
Since the proofs of \eqref{lipschitz stability gamma Sigma 2} and \eqref{pointwise gamma normal Sigma 2} are similar, we prove \eqref{pointwise gamma normal Sigma 2}, assuming that \eqref{lipschitz stability gamma Sigma 2} holds.

%

\begin{eqnarray}\label{Alessandrini 2 gamma2}
\int_{\Sigma} &\left[ \p_{z_n} G_2(\cdot,z)\,\,\sigma^{(1)}(\cdot)\nabla \p_{y_n} G_1(\cdot,y)\cdot\nu - \p_{y_n} G_1(\cdot,y)\sigma^{(2)}(\cdot)\nabla \p_{z_n} G_2(\cdot,z)\cdot\nu\right]\:dS\nonumber\\
&=\partial_{y_n}\partial_{z_n}S_{\mathcal{U}_{1}}(y,z) +\int_{\mathcal{W}_{1}}(\sigma^{(1)}-\sigma^{(2)})(\cdot)\partial_{y_n}\nabla G_1(\cdot,y)\cdot\partial_{z_n}\nabla G_2(\cdot,z).
\end{eqnarray}

Let $\rho_0=r_0\slash C_4$, where $C_4$ is the constant introduced in Theorem \ref{teorema stime asintotiche}. Pick $r\in(0,r_0/6)$. Fix the point $w=w(P_2)=P_2-\tau\nu(P_2)$ where $\tau=a^{\bar{h}-1}\lambda_1$.
We split the integral solution into two parts:
\begin{equation}\label{S=I1+I2}
\partial_{y_n}\partial_{z_n}S_{\mathcal{U}_1}(w,w)=I_1(w)+I_2(w),
\end{equation}

where

\[I_1(w)=\int_{B_{\rho_0}(P_2)\cap D_2}(\gamma_2^{(1)}-\gamma_2^{(2)})(\cdot)\:A(\cdot)\:\partial_{y_n}\nabla G_1(\cdot,w)\cdot
\partial_{z_n}\nabla G_2(\cdot,w),\]

\[I_2(w)=\int_{\mathcal{U}_2\setminus (B_{\rho_0}(P_2)\cap
	D_2)}(\sigma^{(1)}-\sigma^{(2)})(\cdot)\:\partial_{y_n}
\nabla G_1(\cdot,w)\cdot
\partial_{z_n}\nabla G_2(\cdot,w).\]

As in the boundary estimates, we can bound from above $I_2(w)$ as follows:

\begin{equation}\label{stima I2}
|I_2(w)|\leq CE\rho_0^{-n}.
\end{equation}

Now, let us estimate from below the integral $I_1(w)$ in terms of the quantity $|\partial_{\nu}(\sigma^{(1)}_2-\sigma^{(2)}_2)(P_2)|$. First, notice that for any $x\in B_{\rho_0}(P_2)\cap\Sigma_2$ we can rewrite $\gamma^{(i)}_2$ as
\begin{equation}\label{expansion}
\gamma^{(i)}_2(x)=\gamma^{(i)}_2(P_2)+D_T\gamma^{(i)}_2(P_2)\cdot(x-P_2)'+\p_{\nu}(\gamma^{(i)}_2(P_2))(x-P_2)_n.
\end{equation}
By \eqref{expansion},
\begin{eqnarray*}
	|I_1(w)| & &\geq \left|\int_{B_{\rho_0}(P_1)\cap D_2}(\partial_{\nu}(\gamma^{(1)}_2-\gamma^{(2)}_2)(P_2))(x-P_2)_n\:A(x)\:\partial_{y_n}\nabla G_1(\cdot,w)\cdot
	\partial_{z_n}\nabla G_2(\cdot,w)\right|\nonumber\\
	& &-\int_{B_{\rho_0}(P_2)\cap D_2}|(D_T(\gamma^{(1)}_2-\gamma^{(2)}_2)(P_2))\cdot (x-P_2)'|\:|A(x)\:\partial_{y_n}\nabla G_1(\cdot,w)\cdot
	\partial_{z_n}\nabla G_2(\cdot,w)|\nonumber\\
	& &-\int_{B_{\rho_0}(P_2)\cap D_2}|(\gamma^{(1)}_2-\gamma^{(2)}_2)(P_2)|\:|A(x)\:\partial_{y_n}\nabla G_1(\cdot,w)\cdot
	\partial_{z_n}\nabla G_2(\cdot,w)|.
\end{eqnarray*}

Up to a rigid transformation of coordinates, we can assume that $P_2$ coincides with the origin $O$ of the coordinate system. By Theorem \ref{teorema stime asintotiche},

\begin{eqnarray}\label{stima S}
|I_1(w)|
& &\geq
|\partial_{\nu}(\gamma^{(1)}_2-\gamma^{(2)}_2)(O)|C\int_{B_{\rho_0}\cap
	D_2}|\partial_{y_n}\nabla_x\Gamma(Jx,Jw)|^2\: |x_n|\nonumber\\
& & -C E\int_{B_{\rho_0}\cap
	D_2}|\partial_{y_n}\nabla_x\Gamma(Jx,Jw)|\:|x-w|^{\theta_2-n}|x_n|\nonumber\\
& &-C E\int_{B_{\rho_0}\cap
	D_2}|x-w|^{2\theta_2-2n}|x_n|\nonumber\\
& &-\int_{B_{\rho_0}\cap D_2}|D_T(\gamma^{(1)}_2-\gamma^{(2)}_2)(O)|\:|x'|\:|A(x)\:\partial_{y_n}\nabla G_1(\cdot,w)\cdot
\partial_{z_n}\nabla G_2(\cdot,w)|\nonumber\\
& & - \int_{B_{\rho_0}\cap D_2}\!\!\!|(\gamma^{(1)}_2-\gamma^{(2)}_2)(0)|\:|A(x)\:\partial_{y_n}\nabla G_1(\cdot,w)\cdot
\partial_{z_n}\nabla G_2(\cdot,w)|.
\end{eqnarray}

We can estimate the two last terms of the right hand side by \eqref{lipschitz stability gamma Sigma 2}. Then
\begin{eqnarray*}
	|I_1(w)| &\geq &
	|\partial_{\nu}(\gamma^{(1)}_2-\gamma^{(2)}_2)(O)|C\int_{B_{\rho_0}\cap
		D_2}|x-w|^{1-2n}\noindent\\
	&-&C E\int_{B_{\rho_0}\cap
		D_2}|x-w|^{\theta_2+1-2n}\noindent\\
	&-&C E\int_{B_{\rho_0}\cap
		D_2}|x-w|^{2\theta_2+1-2n}\nonumber\\
	&-&(\varepsilon_0+E)\left(\omega_{1/C}^{(3)}\left(\frac{\varepsilon_0}{\varepsilon_0+E}\right)\right)^{1/C}
	\int_{B_{\rho_0}\cap D_2}\:|x-w|^{1-2n}\nonumber\\
	&-&(\varepsilon_0+E)\left(\omega_{1/C}^{(3)}\left(\frac{\varepsilon_0}{\varepsilon_0+E}\right)\right)^{1/C}\int_{B_{\rho_0}\cap D_2}\:|x-w|^{-2n},
\end{eqnarray*}
where the constant $C>0$ depends on the \textit{a-priori} data and on $J$. This leads to

\begin{equation*}
\left|\partial_{\nu}(\gamma^{(1)}_2-\gamma^{(2)}_2)(O)\right|\overline{r}^{(1-n)}\le |I_1(w)| + C\left\{(\varepsilon_0+E)\left(\omega_{1/C}^{(3)}\left(\frac{\varepsilon_0}{\varepsilon_0+E}\right)\right)^{1/C} \tau^{-n} + E\frac{\tau^{1-n+\theta_2}}{\rho_0^{\theta_2}}\right\}.
\end{equation*}

Secondly, by \eqref{S=I1+I2} and \eqref{stima I2},

\begin{equation*}
|I_1(w)|\le |\partial_{y_n}\partial_{z_n}S_{\mathcal{U}_{1}}(w,w)| + C E \rho_0^{-n}.
\end{equation*}

Combining the last two inequalities, it follows that

\begin{eqnarray}
\left|\partial_{\nu}(\gamma^{(1)}_2-\gamma^{(2)}_2)\right|\tau^{(1-n)} &\le& |\partial_{y_n}\partial_{z_n}S_{\mathcal{U}_{1}}(w,w)| + C\bigg\{E\rho_0^{-n}\nonumber\\
&+&(\varepsilon_0+E)\left(\omega_{1/C}^{(3)}\left(\frac{\varepsilon_0}{\varepsilon_0+E}\right) \right)^{1/C}\tau^{-n}+
+E\frac{\tau^{1-n+\theta_2}}{\rho_0^{\theta_2}}\bigg\}.\nonumber
\end{eqnarray}

By unique continuation (Proposition \ref{proposizione unique continuation finale}), we can estimate the integral solution as	
\[
\left|\partial_{y_j}\partial_{z_i} S_{\mathcal{U}_1}(w,w)\right|
\leq
r_0^{{-n}}C^{\bar{h}}(\varepsilon_0 +\delta_1+E)
\left(\omega^{(2)}_{1/C}\left(\frac{\varepsilon_0+\delta_1}{E+\delta_1+\varepsilon_0}\right)\right)^{\left(1/C\right)^{{{\bar{h}}}}},\]
so that 

\begin{eqnarray}
\left|\partial_{\nu}(\gamma^{(1)}_2-\gamma^{(2)}_2)(O)\right| &\le & C^{\bar{h}}(\varepsilon_0 +\delta_1+E)
\bigg(\omega^{(2)}_{1/C}\left(\frac{\varepsilon_0+\delta_1}{E+\delta_1+\varepsilon_0}\right)\bigg)^{\left(1/C\right)^{{{\bar{h}}}}}\tau^{(n-1)}+\nonumber\\
&+& C\tau^{(-1)}(\varepsilon_0+E)\left(\omega_{1/C}^{(3)}\left(\frac{\varepsilon_0}{\varepsilon_0+E}\right)\right)^{1/C}+CE\frac{\tau^{\theta_2}}{\rho_0^{\theta_2}}.\nonumber\\
\end{eqnarray}
Since $\hh$ is a function of $r$, we have to estimate $C^{\bar{h}}$ and $\Big(\frac{1}{C}\Big)^{\bar{h}}$ in terms of $r$. Recalling \eqref{Sh2}, it turns out that

\begin{eqnarray}
\left(\frac{d_1}{r}\right)^{C_1}\le C^{\bar{h}}\le C_2 \left(\frac{d_1}{r}\right)^{C_1} \nonumber .
\end{eqnarray}
Since $\tau\leq \lambda_1\cdot\frac{r}{d_1}$,

\begin{eqnarray}\label{461}
|\partial_{\nu}(\gamma^{(1)}_2-\gamma^{(2)}_2)(O)| &\leq & C(\varepsilon_0 +E)\bigg\{\left(\frac{r}{d_1}\right)^{n-1-C}
\bigg(\omega^{(2)}_{1/C}\left(\frac{\varepsilon_0 +\delta_1}{E+\delta_1+\varepsilon_0}\right)\bigg)^{\left(\frac{r}{d_1}\right)^C}+\nonumber\\
&+&\left(\frac{r}{d_1}\right)^{-1} \left(\omega_{1/C}^{(3)}\left(\frac{\varepsilon_0}{\varepsilon_0+E}\right)\right)^{1/C}
+\left(\frac{r}{d_1}\right)^{\theta_2}\bigg\}.
\end{eqnarray}
One can show that the following inequality holds:

\begin{equation}\label{iteration 1}
\frac{\varepsilon_0 +\delta_1}{E+\delta_1+\varepsilon_0}\leq C \omega_{1/C}^{(0)}\left(\frac{\varepsilon_0}{\varepsilon_0+E}\right).
\end{equation}
Then, combining \eqref{iteration 1} together with \eqref{461}, 

\begin{equation*}
\left|\partial_{\nu}(\gamma^{(1)}_2-\gamma^{(2)}_2)(O)\right| \leq C(\varepsilon_0 +E)\bigg\{\left(\frac{r}{d_1}\right)^{n-1-C}
\bigg(\omega^{(3)}_{1/C}\left(\frac{\varepsilon_0}{E+\varepsilon_0}\right)\bigg)^{\left(\frac{r}{d_1}\right)^C}+\left(\frac{r}{d_1}\right)^{\theta_2}\bigg\}.
\end{equation*}
Finally, optimizing with respect to $r$, \eqref{pointwise gamma normal Sigma 2} follows.

Proceeding as above, for $k=3,\dots,K$, one can show that the following inequalities hold:
\begin{eqnarray}
\| \gamma^{(1)}_k-\gamma^{(2)}_k\|_{L^{\infty}(\Sigma_k\cap B_{r_1}(P_k))}&\leq& C(\varepsilon_0+E)\left(\omega_{1/C}^{(2k-1)}\left(\frac{\varepsilon_0}{\varepsilon_0+E}\right)\right)^
{\frac{1}{C}}\!\!\!,\label{lipschitz stability gamma Sigma k}\\
\left|\partial_{\nu}(\gamma^{(1)}_k-\gamma^{(2)}_k)(P_k)\right| &\leq&  C(\varepsilon_0+E)\left(\omega_{1/C}^{(2k)}\left(\frac{\varepsilon_0}{\varepsilon_0+E}\right)\right)^{\frac{1}{C}}.\label{pointwise gamma normal Sigma l}
\end{eqnarray}
By reformulating \eqref{Alessandrini 1} and \eqref{Alessandrini 2} as

\begin{align}
\int_{\Sigma} &\left[ G_2(\cdot,z)\,\,\sigma^{(1)}(\cdot)\nabla G_1(\cdot,y)\cdot\nu - G_1(\cdot,y)\sigma^{(2)}(\cdot)\nabla G_2(\cdot,z)\cdot\nu\right]\:dS=\nonumber\\
&=S_{\mathcal{U}_{k-1}}(y,z)+\int_{\mathcal{W}_{k-1}}(\sigma^{(1)}-\sigma^{(2)})(\cdot)\nabla G_1(\cdot,y)\cdot\nabla G_2(\cdot,z)
\end{align}
and
\begin{align}
\int_{\Sigma} &\left[ \p_{z_n} G_2(\cdot,z)\,\,\sigma^{(1)}(\cdot)\nabla \p_{y_n} G_1(\cdot,y)\cdot\nu - \p_{y_n} G_1(\cdot,y)\sigma^{(2)}(\cdot)\nabla \p_{z_n} G_2(\cdot,z)\cdot\nu\right]\:dS\nonumber\\
&=\partial_{y_n}\partial_{z_n}S_{\mathcal{U}_{k-1}}(y,z) +\int_{\mathcal{W}_{k-1}}(\sigma^{(1)}-\sigma^{(2)})(\cdot)\nabla\partial_{y_n}G_1(\cdot,y)\cdot\nabla\partial_{z_n}G_2(\cdot,z),
\end{align}
respectively, the procedure is similar to the one seen above. We just point out that, for $(y,z)\in \mathcal{W}_k \times \mathcal{W}_k$,
\begin{equation*}
|S_{\mathcal{U}_{k-1}}(y,z)|
\leq Cr_0^{2-n}(\varepsilon_0+\delta_{k-1}),
\end{equation*}
then we can bound from above the integral solution by unique continuation \eqref{estim1} and \eqref{estim2}. 

Notice that
\[
\delta_k\leq \delta_{k-1}+\|\gamma^{(1)}_k - \gamma^{(2)}_k  \|_{L^{\infty}(D_k)}.
\]
From the property \eqref{property1} it follows that
\[
\omega^{(2k)}_{1/C}(1)\leq \frac{\varepsilon_0+\delta_{k-1}+E}{\varepsilon_0+\delta_{k-1}} \omega^{(2k)}_{1/C}\left(\frac{\varepsilon_0+\delta_{k-1}}{\varepsilon_0+\delta_{k-1}+E}\right)
\]
and 
\[ \delta_{k-1}+\varepsilon_0\leq (\omega^{(2k)}_{1/C}(1))^{-1}(\varepsilon_0+\delta_{k-1}+E)\left(\omega^{(2k)}_{1/C}\left(\frac{\varepsilon_0+\delta_{k-1}}{\varepsilon_0+\delta_{k-1}+E}\right)\right).
\]
By the estimates \eqref{lipschitz stability gamma Sigma k} and \eqref{pointwise gamma normal Sigma l} it follows that
\[
\delta_k+\varepsilon_0\leq C (\varepsilon_0+E)\Big(\omega_{1/C}^{(2k)}\Big(\frac{\varepsilon_0}{\varepsilon_0+E}\Big)\Big)^{1\slash C}.
\]
This leads to the following estimate for $E=\delta_K$

\begin{equation*}
E+\varepsilon_0\leq C(\varepsilon_0 + E)\left(\omega_{1/C}^{(2K)}\left(\frac{\varepsilon_0}{\varepsilon_0 + E}\right)\right)^{\frac{1}{C}}.
\end{equation*}
%
%
Since the function $\omega_{1/C}$ is invertible for $\frac{\varepsilon_0}{\varepsilon_0 + E}<e^{-2}$ (otherwise the statement is proven), it follows that

\[E\leq \frac{C-\Big(\omega_{1\slash C}^{(2K)}\left(\frac{1}{C}\right)\Big)^{-1}}{\Big(\omega_{1\slash C}^{(2K)}\left(\frac{1}{C}\right)\Big)^{-1}}\:\varepsilon_0.\]

Hence, \eqref{stabilita'globale gamma} is proven.

\end{proof}
\begin{proof}[Proof of Corollary \ref{corollary}]
	Assume that the hypothesis of Theorem \ref{teorema principale} hold, then there exists a constant $C>1$ such that
	\[
	\|\sigma^{(1)}-\sigma^{(2)}\|_{L^{\infty}(\Omega)}\leq C
	\left(\mathcal{J}(\sigma^{(1)},\sigma^{(2)})\right)^{1\slash 2}.
	\]
	First, by the Alessandrini's identity,
	\[
	S_{\mathcal{U}_0}(y,z)=\la (\Lambda_1-\Lambda_2)G_1(\cdot,y),G_2(\cdot,z)\ra,
	\]
	where $G_1(\cdot,y), G_2(\cdot,z)\in H^{1\slash 2}_{00}(\Sigma)$ for $y,z\in D_0$ since they are weak solutions to the problem \eqref{greensystem}. Then, it follows that
	\begin{equation*}
	|S_{\mathcal{U}_0}(y,z)|\leq C \|\Lambda_1-\Lambda_2\|_*,
	\end{equation*}
	where
	\[
	\|\Lambda_1-\Lambda_2\|_*=\sup\limits_{f,g\in H^{1\slash 2}_{00}(\Sigma), \:\|g\|=\|\varphi\|=1} |\la(\Lambda_1-\Lambda_2) g,\varphi\ra|.
	\]
	Then
	\begin{equation}\label{stabilita'globale gamma}
	\Big(\mathcal{J}(\sigma^{(1)},\sigma^{(2)})\Big)^{1\slash 2}\leq C \|\Lambda_1-\Lambda_2\|_*,
	\end{equation}
	where $C>0$ depends on the \textit{a-priori} data only. Then the inequality \eqref{corollaryequation} trivially follows.
	
\end{proof}


\section{Proof of technical propositions}\label{technical}

In this section we give the proof of the propositions needed for the proof of the main result (Theorem \ref{teorema principale}).


\subsection{Asymptotic estimates}
Let $0<\mu<1$ and $B^+\in C^{\mu}(Q^+_r)$, $B^-\in C^{\mu}(Q^-_r)$ be symmetric, positive definite, matrix valued functions and define
\[
B(x)=
\begin{cases}
	B^+(x), &x\in Q^+_r,\\
	B^-(x), &x\in Q^-_r,
\end{cases}
\]
such that $B$ satisfies the uniform ellipticity condition
\begin{displaymath}
	\begin{array}{ll}
		\lambda^{-1}_0|\xi|^2\,\, \leq\,\, B(x)\xi\cdot\xi\,\,\leq\,\, \lambda_0 |\xi|^2, &\textrm{$\textnormal{for a.e. }x\in Q_r$, $\textnormal{for every }\xi\in\R^n$},
	\end{array}
\end{displaymath}
where $\lambda_0>0$ is a constant. Let $\bar{b}>0$ and define 
\begin{displaymath}
	b(x)=\left\{ \begin{array}{ll}
		b^+ + B^+\cdot x, &\quad x\in Q^+_{r},\\
		b^- + B^-\cdot x, &\quad x\in Q^-_{r} \ ,
	\end{array} \right.
\end{displaymath}
where $b^+, b^-\in \mathbb{R}, B^+, B^-\in \mathbb{R}^n$ and $0<\bar{b}^{-1}\leq b(x)\leq \bar{b}$.
\begin{theorem}\label{rego}
	Let $r>0$ be a fixed number. Let $b(x)$ and $B(x)$ be as above. Let $U\in H^1(Q_r)$ be a solution to
	\begin{equation*}
	\mbox{div}(b(x)\,B(x)\,\nabla U)=0\ ,\quad \mbox{in }Q_r.
	\end{equation*}
	Then, there exist positive constants $0<\alpha'\le 1, C>0$ depending on $\bar{b},r,\lambda_0$ and $n$ only, such that for any $\rho\le \frac{r}{2}$ and
	for any $x\in Q_{r-2\rho}$, the following estimate holds
	\begin{equation}\label{stimareg}
	\|\nabla U\|_{L^{\infty}(Q_{\rho}(x))}+ {\rho}^{{\alpha}'}|\nabla U|_{\alpha',Q_{\rho}(x)\cap Q^+_{r} } + {\rho}^{{\alpha}'}|\nabla U|_{\alpha',Q_{\rho}(x)\cap Q^-_{r} }
	\le \frac{C}{\rho^{1+n/2}}\|U\|_{L^2(Q_{2\rho}(x))}.
	\end{equation}
	
\end{theorem}

\begin{proof}
	For the proof we refer to Li-Vogelius \cite{Li-Vo}, where piecewise $C^{1,\alpha'}$
	estimates for solutions to elliptic equations in divergence form with
	piecewise H\"{o}lder continuous coefficients have been demonstrated.
\end{proof}

\begin{proof}[Proof of Theorem \ref{teorema stime asintotiche}]
	
	Let us consider a conductivity $\sigma$ of the form
	\[
	\sigma(x) = \sum_{k=1}^{N} \gamma_k(x)\,\chi_{D_k}(x)\, A(x).
	\]
	
	First, fix $k=1,\dots,K$. Up to a rigid transformation, we the point $P_{k+1}$ can be identified with the origin and $\gamma_{k}(0)=\gmeno$ and $\gamma_{k+1}(0)=\gpiu$ for $k\in\R$. For any $x=(x',x_n)$, denote $x^*=(x',-x_n)$. 
	
	Let us introduce a linear change of coordinates 
	\begin{align*}
	L:&\R^n\rightarrow \R^n\\
	&\xi \mapsto L\xi:= R\,J \xi,
	\end{align*}
	where $J=\sqrt{A^{-1}(0)}$ and the matrix $R$ is orthogonal and represents the planar rotation in $\R^n$ that rotates the unit vector $\frac{v}{\|v\|}$, where $v=\sqrt{A(0)}e_n$ to the $n$th standard unit vector $e_n$ and such that
	\[
	R|_{(\pi)^{\perp}}=Id|_{(\pi)^{\perp}},
	\]
	where $\pi$ is the plane generated by $e_n$ and $v$ and $(\pi)^{\perp}$ is the orthogonal complement of $\pi$ (see \cite{G-S}). Moreover, the following relations hold
	\begin{itemize}
		\item $A(0)=L^{-1}\cdot (L^{-1})^T$,
		\item $(L\xi)\cdot e_n =\frac{1}{\|v\|}\xi\cdot e_n$,
		\item $\sigma_{A(0)}(\xi)=L^{-1}\sigma_I(L\xi)(L^{-1})^T$, where $\sigma_I(L\xi)=\sigma_I(x)= (\gmeno+(\gpiu-\gmeno)\chi^+(x))I$.
	\end{itemize}
	
	A fundamental solution of the operator $\mbox{div}_{\xi}((\gmeno+(\gpiu-\gmeno)\chi^+(\cdot))A(0)\nabla_{\xi}\cdot)$ has the following explicit form
	\begin{equation}\label{fund}
	\begin{array}{ll}
	H_{A(0)}(\xi,\eta)= \left\{
	\begin{array}{lcl} 
	|J| \Big(\frac{1}{\gpiu}\Gamma(L\xi,L\eta) + \frac{\gpiu-\gmeno}{\gpiu(\gpiu+\gmeno)}\Gamma(L\xi,L^*\eta)\Big) , &\textrm{$\textnormal{if }\ \xi_n,\eta_n>0$},\\
	|J| \Big(\frac{2}{\gpiu+\gmeno}\Gamma(L\xi,L\eta)\Big) , &\mbox{if}\ \xi_n \eta_n<0,\ \\
	|J| \Big(\frac{1}{\gmeno}\Gamma(L\xi,L\eta) + \frac{\gmeno-\gpiu}{\gmeno(\gpiu+\gmeno)}\Gamma(L\xi,L^*\eta)\Big) ,
	&\textrm{$\textnormal{if }\ \xi_n, \eta_n<0$},
	\end{array}
	\right.
	\end{array}
	\end{equation}
	where $|J|$ denotes the determinant of the matrix $J$ and $L^*$ is the matrix whose coefficients follow the rule
	\[
	l^*_{ij}=l_{ij}, \quad \mbox{for }i=1,\dots,n-1, j=1,\dots,n,\qquad l^*_{nj}=-l_{nj} \quad \mbox{for }j=1,\dots,n.
	\]
	Set $H(\xi,\eta)=H_{A(0)}(\xi,\eta)$. Denote with $\widetilde{\Omega}$ the augmented domain obtained after having performed the change of coordinates $L$. Define the distribution
	\begin{eqnarray}\label{diff}
	R(\xi,\eta)=G(\xi,\eta) - H(\xi,\eta),
	\end{eqnarray}
	where $G(\cdot,\eta)$ is the weak solution to \eqref{greensystem}, then $R(\xi,\eta)$ is a weak solution to the following boundary value problem
	\begin{displaymath}
	\left\{ 
	\begin{array}{ll}
	\mbox{div}_{\xi}\big(\sigma(\cdot)\nabla R(\cdot,\eta)\big)=-\mbox{div}_{\xi}\big((\sigma(\cdot)-\sigma_0(\cdot))\nabla_{\xi} H(\cdot,\eta)\big), &\mbox{in }\widetilde{\Omega},\\
	R(\cdot,\eta)=-H(\cdot,\eta), &\mbox{on }\p\widetilde{\Omega},
	\end{array}
	\right.
	\end{displaymath}
	where $\sigma_0(\cdot)=(\gmeno+(\gpiu-\gmeno)\chi^+(\cdot))A(0)$. By the representation formula over $\widetilde{\Omega}$, it follows that $R$ satisfies the following integral identity
	\begin{eqnarray}\label{reprformR}
	&&\  R(\xi,\eta)=\ - \int_{\widetilde{\Omega}}(\sigma(\zeta)-\sigma_0(\zeta))\nabla_{\zeta}H(\zeta,\eta)\cdot\nabla_{\zeta}{G}(\zeta,\xi)d\zeta + \int_{\partial \widetilde{\Omega}}\sigma(\zeta)\nabla G(\zeta,\xi)\cdot \nu \, H(\zeta,\eta)dS(\zeta).
	\end{eqnarray}
	The integral over $\p\widetilde\Omega$ at the right hand side of \eqref{reprformR} can be easily bounded from above as in  \cite[Equation (4.10)]{A-dH-G-S} by a constant $C>$ which depends on the \textit{a-priori} data only.
	
	Set $\gamma_0(\cdot)=\gmeno+(\gpiu-\gmeno)\chi^+(\cdot)$. Locally, in a neighbourhood of the origin, the following estimate holds
	\begin{align}\label{gg0}
		|\sigma(\zeta)-\sigma_0(\zeta)|\le |\gamma(\zeta)A(\zeta)-\gamma_0(\zeta)A(0)|\le  |\gamma(\zeta)|\,|A(\zeta)-A(0)| + |\gamma(\zeta)-\gamma_0(\zeta)|\,|A(0)|\le C\,|\zeta|,
	\end{align}
where $C>0$ depends on $\bar{\gamma}, \bar{A}$ only.
	Moreover by \eqref{boundgreen} we find the following two pointwise bounds:
	\begin{eqnarray*}
		&& |\nabla_{\zeta} G(\zeta, \xi)|\le C |\zeta - \xi|^{1-n}\ \ \mbox{for every}\ \ \zeta,\xi \in Q_{r_0}\ , \\
		&&|\nabla_{\zeta} H(\zeta, \eta)|\le C |\zeta - \eta|^{1-n}\ \mbox{for every}\ \ \zeta, \eta \in Q_{r_0} \ ,
	\end{eqnarray*}
	which together with \eqref{gg0} leads to
	\begin{equation}\label{2aR}
		\left|\int_{\widetilde{\Omega}}(\sigma(\zeta)-\sigma_0(\zeta))\nabla_{\zeta}H(\zeta,\eta)\cdot\nabla_{\zeta}{G}(\zeta,\xi)\:d\zeta\right| \le {C_1}|\xi-\eta|^{3-n-\alpha},
	\end{equation}
for any $0<\alpha<1$. In conclusion, for $\xi\in B^+_{{r_0}}$, $\eta=\eta_ne_n$ with $\eta_n\in(-{r_0},0)$,
	\begin{eqnarray}\label{2R}
		|R(\xi,\eta)| \le {C}|\xi-\eta|^{3-n-\alpha}.
	\end{eqnarray}
	We focus on the estimate for $\nabla _{\xi} R(\xi,e_n\eta_n)$. Fix $\xi\in B^+_{r_0/4}$ and $\eta_n\in(-r_0/4, 0)$, consider the cylinder
	$Q=B'_{h/4}(\xi')\times \left (\xi_n, \xi_n +\frac{h}{4} \right)$.
	where $h=|\xi-\eta|$\ . Notice that  $Q\subset Q^{+}_{\frac{r_0}{2}}$, $Q\subset Q_{\frac{h}{2}}(\xi)$ and $\xi\in \partial Q$.

	By Theorem \ref{rego} it follows that
	\begin{eqnarray}\label{hold}
			|\nabla_{\xi} G(\cdot, e_n\eta_n)|_{\alpha',Q}\ , \ |\nabla_{\xi} H (\cdot, e_n\eta_n)|_{\alpha',Q} \le C  h^{-\alpha' +1-n} \ .
	\end{eqnarray}
		
	Hence by \eqref{diff} and \eqref{hold} we 
	\begin{eqnarray}\label{holdR}
		|\nabla_{\xi} R(\cdot, e_n\eta_n)|_{\alpha',Q}\le C  h^{-\alpha' +1-n}.
	\end{eqnarray}
	From the following interpolation inequality
	\begin{equation*}
		\| \nabla _{\xi} R(\cdot, e_n\eta_n)\|_{L^{\infty}(Q)}\le C \left(\| R(\cdot, e_n\eta_n)\|^{\alpha'/1+\alpha'}_{L^{\infty}(Q)}\left|\nabla_{\xi} R(\cdot, e_n\eta_n) \right|^{1/1+\alpha'}_{\alpha',Q}+\frac{1}{h}\|R(\cdot,\eta_ne_n)\|_{L^{\infty}(Q)}\right),
	\end{equation*}
	together with \eqref{2R} we obtain
	\begin{eqnarray*}
		|\nabla_{\xi} R(\cdot,\eta_ne_n)|\le C h ^{\theta_1 +1-n},
	\end{eqnarray*}
	where $\theta_1= \frac{\alpha'(1-\alpha)}{1+\alpha}$\ .
		
	Now, we look for a pointwise bound for $\nabla_{\eta}\nabla_{\xi} R(\xi,\eta)$. Define the cylinder $\hat{Q}= B'_{\frac{h}{8}}(0)\times \left(\eta_n- \frac{h}{8}, \eta_n \right)$. As before, we have that $\hat{Q}\subset Q{^-}_{\frac{r_0}{4}},\hat{Q}\subset Q_{\frac{h}{4}}(\eta)$ and $\xi\notin  Q_{\frac{h}{4}}(\eta)$.
		
	Let $k$ be an integer such that $k\in \{1,\dots, n\}$. Notice that $\partial_{\xi_k}\Gamma(\xi, \cdot)$ is a weak solution to the Laplace equation
	\[
	\Delta_{\eta} (\partial_{\xi_k}\Gamma(\xi,\cdot))=0\ \ \ \mbox{in}\ \ Q_{\frac{h}{4}}(\eta)\ ,
	\] 
	and $\partial_{\xi_k}G(\xi,\cdot)$ is a weak solutions to the problem
	\begin{equation*}
	\left\{ 
	\begin{array}{ll}
	\mbox{div}(\sigma(\cdot)\nabla \p_{\xi_k}G_i(\xi,\cdot))=-\delta(\xi-\cdot) &\mbox{in}\ \ Q_{\frac{h}{4}}(\eta),\\
	G_i(\xi,\cdot)=0 &\mbox{on }\p\Omega.
	\end{array}
	\right.
	\end{equation*}
	By Theorem \ref{rego}, it follows that
	\begin{eqnarray}\label{1s}
		|\nabla_{\eta}\partial_{\xi_k}G(\xi,\cdot)|_{\alpha', \hat{Q}}\le C  h^{-\alpha'-1-\frac{n}{2}}\|\partial_{\xi_k}G(\xi,\cdot) \|_{L^2(Q_{\frac{h}{4}}(\eta))}.
	\end{eqnarray}
	Fix $\bar{\eta} \in Q_{\frac{h}{4}}(\eta)$, then $\bar\eta\notin Q_{\frac{h}{16}}(\xi)$. By Theorem \ref{rego}, it follows that
	\begin{eqnarray}\label{2s}
		\|\nabla_{\xi} G(\cdot, \bar\eta)\|_{L^{\infty}(Q_{\frac{h}{32}}(\xi))}\le C h^{-1-\frac{n}{2}}\|G(\cdot,\bar\eta)\|_{L^{\infty}(Q_{\frac{h}{16}}(\xi))}\le C h ^{1-n}.
	\end{eqnarray}
	From \eqref{1s} and \eqref{2s} it follows that
	\begin{eqnarray}\label{1h}
		|\nabla_{\eta}\partial_{\xi_k}G(\xi,\cdot)|_{\alpha', \hat{Q}} \le C h^{-\alpha'-n}\ \ .
	\end{eqnarray}
	By the representation formula for $\Gamma$, 
	\begin{eqnarray}\label{2h}
		|\nabla_{\eta}\partial_{\xi_k}{\Gamma}(\xi,\cdot)|_{\alpha', \hat{Q}} \le C h^{-\alpha'-n},
	\end{eqnarray}
	and by \eqref{1h} and \eqref{2h}, 
	\begin{eqnarray}\label{2e}
		|\nabla_{\eta}\partial_{\xi_k}R(\xi,\cdot)|_{\alpha', \hat{Q}} \le C h^{-\alpha'-n}\ .
	\end{eqnarray}
	Arguing as above, the following estimate holds:
	\begin{eqnarray}\label{2ee}
		\|\partial_{\xi_k}R(\xi,\cdot)\|_{L^{\infty}(\hat{Q})}\le C h^{\theta_1 +1-n}.
	\end{eqnarray}
	By the following interpolation inequality
	\begin{eqnarray}
		\|\nabla_{\eta}\partial_{\xi_k}R(\xi,\cdot)\|_{L^{\infty}(\hat{Q})}\le C \|\partial_{\xi_k}R(\xi,\cdot) \|^{\frac{\alpha'}{\alpha'+1}}_{L^{\infty}(\hat{Q})}|\nabla_y\partial_{\xi_k}R(\xi,\cdot)|^{\frac{1}{\alpha'+1}}_{\alpha', \hat{Q}}
	\end{eqnarray}
	and by \eqref{2ee} and \eqref{2e}, we conclude that
	\begin{eqnarray}
		|\nabla_{\eta}\partial_{\xi_k}R(\xi,\eta)|\le C h^{\theta_2-n},
	\end{eqnarray}
	where $\theta_2=\frac{\theta_1\alpha'}{1+\alpha'}$\ .	
\end{proof}

\subsection{Propagation of smallness}
	In order to prove Theorem \ref{proposizione unique continuation finale}, we state and prove a preliminary Proposition \ref{preliminaryprop}, where we determine a pointwise bound for the weak solution to the conductivity equation in the interior of $\widetilde\Omega$.
	\begin{proposition}\label{preliminaryprop}
		Let $v\in H^1(\widetilde{\Omega})$ be a weak solution to
		\begin{equation}\label{eq Wk}
		\mbox{div}(\sigma\,\nabla v)=0 \qquad \mbox{in }\mathcal{W}_k,
		\end{equation}
		where $k\in\{0,\dots,K-1\}$. Suppose there exist $E,\epsilon>0$ such that
		\begin{eqnarray}\label{first}
		&|v(x)|\leq r_0^{2-n}\epsilon\qquad &\forall x\in D_0,
		\end{eqnarray}
		\begin{eqnarray}\label{second}
		&|v(x)|\leq E\big(r_0 d(x)\big)^{1-(n\slash 2)}\qquad &\forall x\in \mathcal{W}_k
		\end{eqnarray}
		
		Then, for every $r\in (0,d_1]$,
		\begin{equation}
		|v(w_{\hh}(P_{k+1}))|\leq r_0^{2-n}C^{\hh}(E+\epsilon)\Big(\omega^{(K)}_{1\slash C}\Big(\frac{\epsilon}{\epsilon+E}\Big)\Big)^{(1\slash C)^{\hh}}
		\end{equation}
		where $C>1$ depends only on a-priori data.
	\end{proposition}
	\begin{proof}[Proof of Proposition \ref{preliminaryprop}]
			We adapt the proof in \cite[Proposition 4.4]{A-V} to the case of the anisotropic conductivity.
			
			To begin with, we introduce some parameters. Recall from \eqref{Wk} that $\mathcal{W}_k=\bigcup_{m=0}^k D_m$, then for the domain index $m\in \{0,\dots,K-1\}$,
			\begin{eqnarray}
			&r_l=\frac{r_0}{l}, \quad &\overline{\rho}=\frac{r_l}{32l\sqrt{1+L^2}},\\
			&y_{m+1}=P_{m+1}-\frac{r_l}{32}\nu(P_{m+1}),\quad &\tilde{y}_{m+1}=P_{m+1}+\frac{r_l}{32}\nu(P_{m+1}),\\
			&v_m=v|_{D_m},
			\end{eqnarray}
			where $P_{m+1}$ and $\nu(P_{m+1})$ have been defined in subsection \ref{subsection Green function}. We claim that for every $m\in \{0,\dots,K-1\}$,
			\begin{equation}\label{claim1}
			\|v\|_{L^{\infty}(B_{\overline{\rho}}(\tilde{y}_{m+1}))}\leq r^{2-n}_0C^{m+1}(E+\epsilon)\Bigg(\omega^{(m+1)}_{1\slash C}\Big(\frac{\epsilon}{\epsilon+E}\Big)\Bigg)
			\end{equation}
 and prove \eqref{claim1} by induction as follows.
			
			\textit{Case $m=0$.}
				\\
				
				Up to a rigid transformation of coordinate, we can suppose that $y_1=-\frac{r_l}{32} e_n$. From \eqref{second},
				\begin{eqnarray}\label{estimate}
				&\|v\|_{L^{\infty}(D_0)}&\leq E\Big(r\sup_{x\in D_0}d(x)\Big)^{1-n\slash 2}.
				\end{eqnarray}

				Choose an arbitrary point $\yy\in \Sigma_{1}$, possibly different from $P_1$. Let $\phi$ be a Jordan curve joining $y_1$ to $w_1(\yy)$ such that $\phi\subset (D_0)_{\bar{d}}$, where $\bar{d}=\min\{\mbox{dist}(y_1,\Sigma_{1}),\mbox{dist}(w_1(\yy),\Sigma_{1})\}$, and $(D_0)_{\bar{d}}$ is connected. Notice that $w_1(\yy)\in (D_0)_{\bar{d}}$. 
				Let us define a set of points $\{\phi_i\}$, $i=1,\dots,s$ through the following process:
				\begin{itemize}
					\item $\phi_1=\phi(0)=y_1$;
					\item for $i>1$, set
					\[
					\phi_{i+1}=
					\begin{cases}
					\phi(t_i), &\text{if }|\phi_i-w_1(\yy)|>2r_l \text{ where }t_i=\max\{t_i: \, |\phi(t)-\phi_i|=2r_l\},\\
					w_1(\yy), &\text{if }|\phi_i-w_1(\yy)|<2r_l \text{ and set }s=i+1.
					\end{cases}
					\]
				\end{itemize}
				Apply the three sphere inequality in the case of pure principal part (see \cite[Theorem 2.1]{A-R-R-V}) on spheres centred at $\phi_1=y_1$ for which estimates \eqref{estimate} and \eqref{first} hold, with suitable rays $r, 3r, 4r$:
				\[
				\|v\|_{L^2(B_{3r}(y_1))}\leq Q\|v\|_{L^2(B_{r}(y_1))}^{\delta} \|v\|_{L^2(B_{4r}(y_1))}^{1-\delta} \leq Qr_0^{2-n}\epsilon^{\delta} E^{1-\delta},
				\]
				where $\delta = \frac{\log\Big(\frac{4\lambda}{3}\Big)}{\log\Big(\frac{4\lambda}{3}\Big)+C\log\Big(\frac{3}{\lambda}\Big)}$ and $Q>1$ is a constant which depends on $\lambda$, $L$, $\max\Big\{\frac{4r}{r_0},1\Big\}$.
				
				Notice that $B_{r}(\phi_{2})\subset B_{3r}(\phi_{1})=B_{3r}(y_1)$ so that the $L^{2}$-norm of $v$ on $B_{r}(\phi_2)$ can be easily estimated applying the three sphere inequality for the spheres of rays $r, 3r, 4r$ centred at $\phi_2$. Moreover, by \cite[Theorem 8.17]{G-T}, since $v$ is a weak solution to \eqref{eq Wk}, it follows that
				\[
				\|v\|_{L^{\infty}(B_{R\slash 2}(y))} \leq C \rho^{n\slash 2} \|v\|_{L^2(B_{R}(y))},
				\]
				where $C$ depends on $n$, $\lambda$ and $|\Omega|$. By iterating this process, we can estimate the $L^{\infty}$-norm of $v$ along the chain of spheres centred at points $\phi_i$ of the curve $\phi$. In conclusion,
				\begin{equation}\label{l2v}
				\|v\|_{L^{\infty}(B_{r}(w_1(\yy)))}\leq \|v\|_{L^{\infty}(B_{3r}(\phi_{s-1}))}\leq Cr^{2-n}\epsilon^{\delta^s} E^{1-\delta^s}.
				\end{equation}
				
				Fix $r\in (0,d_1]$. Recalling the parameters introduced in \eqref{parameters}, the following inclusions hold: 
				\[
				B_{\rho_{k+1}}(w_{k+1}(\yy))\subset B_{3\rho_k}(w_k(\yy))\subset B_{4\rho_k}(w_k(\yy))\subset C\Big(\yy,\nu(\yy),\beta_1,r_0\slash 3\Big),
				\]
				for any $k=1,2,\dots$. Notice that $\rho_1<r_l$ for a suitable $l$, then $B_{\rho_1}(w_{1}(\yy))\subset B_{r_l}(w_1(\yy))$. We proceed by moving from one centre to the successive one along the axis of the cone $C\Big(\yy,\nu(\yy),\beta_1,r_0\slash 3\Big)$ allowing to get closer and closer to the vertex $\yy$ and stop this process when we reach the sphere of radius $\rho_{\hh}$. 
				Then, from \eqref{l2v},
				\begin{equation}\label{estimate2}
				\|v\|_{L^{\infty}(B_{\rho_{\hh}}(w_{\hh}(\yy)))}\leq C \epsilon^{\delta^{s+\hh-1}}E^{1-\delta^{s+\hh-1}}.
				\end{equation}
				
				By the triangular inequality,
				\begin{equation}\label{triangv}
				|v(\yy)|\leq |v(\yy)-v(\yy-r\nu(\yy))|+|v(\yy-r\nu(\yy))|.
				\end{equation}
			First, we estimate the second term on the righthand side of \eqref{triangv}. Since $\yy-re_n\in B_{\rho_{\hh}}(w_{\hh}(\yy))$,
				\begin{align*}
				|v(\yy-r\nu(\yy))|\leq C r_0^{2-n} \epsilon^{\delta^{s+\hh-1}}E^{1-\delta^{s+\hh-1}}
				\leq C r^{2-n}(\epsilon + E)\Big(\frac{\epsilon}{E+\epsilon}\Big)^{1-\delta^{s+\hh-1}}.\nonumber
				\end{align*}
				Secondly, we estimate the first term on the righthand side of \eqref{triangv}. Since $\bar{y}\in\mathcal{W}_k$, by \eqref{second},
				\begin{align*}
				|v(\yy)|\leq C E \Big(r_0 \sup\limits_{x\in D_0}d(x)\Big)^{1-(n\slash2)}\leq Cr_0^{2-n} E.
				\end{align*}
				Hence, by Theorem \ref{rego},
				\begin{eqnarray}
				|v(\yy)-v(\yy-r\nu(\yy))| \leq \|\nabla v\|_{L^{\infty}(Q_{r_0\slash 3})}r
				\leq \frac{c}{r_0^{1+n\slash2}}\|v\|_{L^2(Q_{2r_0\slash 3})} r
				\leq Cr_0^{2-n} (E+\epsilon)\Big(\frac{r}{r_0}\Big).\nonumber
				\end{eqnarray}
				Therefore,
				\begin{align*}
				|v(\yy)|\leq C r_0^{2-n}(E+\epsilon)\Big(\frac{r}{r_0}+\Big(\frac{\epsilon}{E+\epsilon}\Big)^{\delta^{s+\hh-1}}\Big).
				\end{align*}
				Minimizing the righthand side of the last inequality with respect to $r$, the following inequality holds:
				\[
				|v(\yy)|\leq Cr_0^{2-n} (E+\epsilon_1) \Big|\log\Big(\frac{\epsilon}{E+\epsilon}\Big)^{\delta^{s}}\Big|^{-\frac{C}{2|\log \delta|}},
				\]
				for a suitable constant $C>0$. Set $\widetilde{\Sigma}_1=\Sigma_{1}\cap Q_{r_l}(P_1)$.	By the arbitrarity of $\yy$, we obtain
				\begin{equation}\label{F}
				\|v\|_{L^{\infty}(\widetilde{\Sigma}_1)}\leq C r_0^{2-n}(E+\epsilon) \omega_{1\slash C}\Big(\frac{\epsilon}{\epsilon+E}\Big).
				\end{equation}
				In order to prove our claim, we need to estimate the gradient of $v$. Recalling that $v_0 = v|_{D_0}$ and $v_1 = v|_{D_1}$ and $v_0$ is harmonic in $D_0$, from the three sphere inequality applied to $\nabla v_0$ and the results of \cite{Li-Vo}, one can recover the following estimates:
				\begin{equation}\label{G}
				\|\nabla v_0\|_{L^{\infty}(\widetilde{\Sigma}_1)}\leq C r_0^{2-n}(E+\epsilon)\omega_{1\slash C}\Big(\frac{\epsilon}{\epsilon+E}\Big),
				\end{equation}
				and 
				\begin{equation}\label{GT}
				\|\nabla_T v_1\|_{L^{\infty}(\widetilde{\Sigma}_1)}=\|\nabla_T v_0\|_{L^{\infty}(\widetilde{\Sigma}_1)}\leq \|\nabla v_0\|_{L^{\infty}(\widetilde{\Sigma}_1)}\leq C r_0^{2-n}(E+\epsilon) \omega_{1\slash C}\Big(\frac{\epsilon}{\epsilon+E}\Big).
				\end{equation}
				Now we can apply the following estimate due to Trytten \cite{T}:
				\begin{eqnarray}\label{trytten}
				\int_{D_1\cap B_{3r_l\slash 8}(P_1)} |\nabla v_1|^2 \leq &&\frac{c}{r_0} \Bigg(\int_{\widetilde{\Sigma}_1} v_1^2 + r_0^2 \int_{\widetilde{\Sigma}_1} |\nabla v_1|^2\Bigg)^{\delta_1} \times\nonumber\\
				&&\times\,\, \Bigg(\int_{\widetilde{\Sigma}_1} v_1^2 + r_0^2\int_{\widetilde{\Sigma}_1} |\nabla v_1|^2 + r_0 \int_{D_1\cap B_{r_l\slash 4}(P_1)}A |\nabla v_1|^2\Bigg)^{1-\delta_1}.
				\end{eqnarray}
				In order to bound the lefthand side of \eqref{trytten}, we have to estimate the following quantities:
				\begin{itemize}
					\item[i)] $\int_{\tilde{\Sigma}_1} v_1^2$; 
					\item[ii)] $\int_{\tilde{\Sigma}_1} |\nabla v_1|^2$;
					\item[iii)] $\int_{D_1\cap B_{r_l\slash 4}(P_1)}A |\nabla v_1|^2$.
				\end{itemize}
				For i), we can just use \eqref{F}. For ii), since $\nabla v_1 = \nabla_T v_1 + (\nabla v_1\cdot \nu)\nu$, 
				\[
				\int_{\widetilde{\Sigma}_1} |\nabla v_1|^2 \leq \int_{\widetilde{\Sigma}_1} |\nabla v_T|^2 + \int_{\widetilde{\Sigma}_1} |(\nabla v_1\cdot \nu) \nu|^2.
				\]
				The first integral on the righthand side can be estimated using \eqref{GT}. For the other term, one uses the transmission conditions
				\begin{eqnarray}
				A(x)\nabla v_0\cdot\nu = A(x)\nabla v_1\cdot \nu,\qquad &&\mbox{on }\Sigma_{1}.
				\end{eqnarray}
				Then, 
				\begin{equation}\label{G1}
				\|\nabla v_1\|_{L^{\infty}(\widetilde{\Sigma}_1)}\leq Cr_0^{1-n}(E+\epsilon) \omega_{1\slash C}\Big(\frac{\epsilon}{\epsilon+E}\Big).
				\end{equation}
				Finally, iii) follows from standard energy estimates.
				
				From the following trace estimate
				\begin{equation}\label{trace}
				\int_{D_1\cap B_{3r_l\slash 16}(P_1)} v_1^2\leq C \Bigg(r_0\int_{\widetilde{\Sigma}_1}v_1^2 + r_0^2\int_{D_1\cap B_{3r_l\slash 8}(P_1)} |\nabla v_1|^2\Bigg),
				\end{equation}
				\eqref{F}, \eqref{trytten}, \eqref{G1} and \eqref{trace} it follows that
				\begin{equation}
				\| v_1\|_{L^{\infty}(B_{\overline{\rho}}(\tilde{y}_{1}))}\leq Cr_0^{1-n}(E+\epsilon) \omega_{1\slash C}\Big(\frac{\epsilon}{\epsilon+E}\Big).
				\end{equation}
				\textit{Case $m\implies m+1$.} Set
				\[
				\epsilon_m=C^{m+1}r_0^{2-n}(E+\epsilon)\Bigg(\omega^{(m+1)}_{1\slash C}\Big(\frac{\epsilon}{\epsilon+E}\Big)\Bigg).
				\]
				By proceeding as above, we end up with the following inequality
				\begin{equation}
				\| v_1\|_{L^{\infty}(B_{\overline{\rho}}(\tilde{y}_{m+1}))}\leq Cr_0^{1-n}(E+\epsilon_m) \omega_{1\slash C}\Big(\frac{\epsilon_m}{\epsilon_m+E}\Big).
				\end{equation}
				By the properties \eqref{property1} and \eqref{property2} of $\omega_{1\slash C}$, the claim follows. To summarise it, we have proved that for any point close enough to the interface, the $L^{\infty}$-norm of $v$ on a small ball can be bound in terms of the quantities the righthand side of \eqref{first} and \eqref{second}. 
			
			For $m< K-1$ the thesis follows by the inequality \eqref{estimate2}, choosing $\yy = P_{m+1}$. 
			
			For $m=K-1$, by condition \eqref{second}, arguing as in the inequality \eqref{estimate2} and applying the claim, it follows that
			\begin{align*}
			|v\big(w_{\hh}(P_K)\big)|\leq C \big(r_0^{2-n}\epsilon_K\big)^{\delta^{s+\hh-1}}(r_0d_1a^{\hh-1}E)^{1-\delta^{s+\hh-1}}
			&\leq C^{\hh} r_0^{2-n} (\epsilon_K+E) \omega_{1\slash C}\Big(\frac{\epsilon_K}{\epsilon_K+E}\Big)\\
			&\leq C^{\hh} r_0^{2-n} (\epsilon + E) \omega_{1\slash C}^{(K)}\Big(\frac{\epsilon}{\epsilon+E}\Big)^{\big(1\slash C^{\hh}\big)}.
			\end{align*}
	\end{proof}	
	
	\begin{proof}[Proof of Proposition \ref{proposizione unique continuation finale}]
	To begin with, recall that for any $(y,z)\in (D_0)_r\times (D_0)_r$, for $r\in(0,d_1]$, the following bound holds:
	\[
	|S_{\mathcal{U}_k}(y,z)|\leq \|\sigma^{(1)}-\sigma^{(2)}\|_{L^{\infty}(\Omega)}\Big(\mbox{dist}(y,\mathcal{U}_k)\,\mbox{dist}(z,\mathcal{U}_k)\Big)^{1-n\slash 2}.
	\]
	For any $y,z \in B_{\rho_{\bar{h}(r)}}(w_{\bar{h}(r)}(Q_{k+1}))$, we apply Proposition \ref{preliminaryprop} once to $v=S_{\mathcal{U}_k}(\cdot,z)$ and then to $v=S_{\mathcal{U}_k}(y,\cdot)$ to obtain 
	\begin{eqnarray}\label{primastima}
	|S_{\mathcal{U}_k}({y},z)|\le  r_0^{{2-n}}C^{\bar{h}(r)}(E+\varepsilon_0)\left(\omega_{1/C}^{(2k)}\left(\frac{\varepsilon_0}{E+\varepsilon_0}\right)\right)
	^{\left(1/C\right)^{\bar{h}(r)}}.
	\end{eqnarray}
	Hence \eqref{estim1} follows from \eqref{primastima}.
	
	Since $S_{\mathcal{U}_k}(y_1,\dots, y_n,z_1,\dots, z_{n})$ is a weak solution in $D_k\times D_k$ of the elliptic equation
	
	\begin{eqnarray}
	{\mbox{div}}_y (\sigma^{(1)}(y){\nabla}_y S_{\mathcal{U}_k}(y,z)) + {\mbox{div}}_z (\sigma^{(2)}(z){\nabla}_z S_{\mathcal{U}_k}(y,z)) = 0,
	\end{eqnarray}
	for any $i,j=1,\dots, n$ it follows that

	\begin{eqnarray}
		& &\|\partial_{x_i}\partial_{x_j} S_{\mathcal{U}_k}(x_1,\dots, x_n,x_{n+1},\dots, x_{2n})\|_{L^{\infty}(B_{\frac{\rho_{\bar{h}(r)}}{2}}(w_{\bar{h}(r)}(Q_{k+1}))\times B_{\frac{\rho_{\bar{h}(r)}}{2}}(w_{\bar{h}(r)}(Q_{k+1})))}\nonumber\\
		& &\le\frac{C}{{\rho^2_{\bar{h}(r)-1}}}\| S_{\mathcal{U}_k }(x_1,\dots, x_n,x_{n+1},\dots, x_{2n})\|_{L^{\infty}( B_{\rho_{\bar{h}(r)}}(w_{\bar{h}(r)}(Q_{k+1}))\times B_{\rho_{\bar{h}(r)}}(w_{\bar{h}(r)}(Q_{k+1})))}
	\end{eqnarray}
	
	where $x_i=y_i, x_{i+n}=z_i$ for $i=1, \dots, n$.
	
	Moreover, since $d_{\bar{h}(r)-1}>r$, it follows that $r<\frac{d_0}{a\rho_0}\rho_{\bar{h}(r)}$, which in turn leads to
	\begin{eqnarray}\label{Sh}
	\|\partial_{x_i}\partial_{x_j} S_{\mathcal{U}_k}(x_1,\dots, x_{2n})\|_{L^{\infty}(\tilde{Q}_{\frac{{\rho_{\bar{h}(r)}}}{2}}(w_{\bar{h}(r)}(Q_{k+1})))} \le\frac{C}{r^2}\| S_{\mathcal{U}_k}(x_1,\dots, x_{2n})\|_{L^{\infty}(\tilde{Q}_{{{\rho_{\bar{h}(r)}}}}(w_{\bar{h}(r)}(Q_{k+1})))}.
	\end{eqnarray}
	
	By \eqref{Sh2}, it follows that	$r^{-2}\le \left(\frac{a}{r_0}\right)^2\left(\frac{1}{a^2}\right)^{\bar{h}(r)}$, and by combining \eqref{Sh} and the above inequality we get the desired estimate.
	
\end{proof}

\section*{Acknowledgments}
The authors wish to thank the anonymous referees for useful comments and remarks which
improved the presentation of the paper.RG would like to acknowledge the support of the Dipartimento
di matematica e geoscienze, Universit`a degli Studi di Trieste, where the research of
this paper was initiated during her sabbatical leave in 2020. The work of ES was performed
under the PRIN Grant No. 201758MTR2-007.


\end{document}